\documentclass[9pt]{article}
\usepackage{amsgen, amsmath, amsfonts, amsthm, amssymb,enumerate,amscd,graphics, amsfonts,dsfont}
\usepackage[all]{xy}
\newtheorem{defn}{Definition}[section]
\newtheorem{prop}[defn]{Proposition}
\newtheorem{lem}[defn]{Lemma}
\newtheorem{thm}[defn]{Theorem}
\newtheorem{cor}[defn]{Corollary}
\newtheorem{rem}[defn]{Remark}
\newtheorem {conj}[defn]{Conjecture}

\newcommand {\ZZ}{{\mathds Z}}
\newcommand {\NN}{{\mathds N}}

\newcommand {\C}{{\mathds C}}

\newcommand {\Q}{{\mathds Q}}
\newcommand {\R}{{\mathds R}}

\newcommand {\M}{{\mathcal M}}

\newcommand {\BB}{{\mathbf B}}
\newcommand {\bff}{{\mathbf f}}
\newcommand {\BF}{{\mathbf F}}

\newcommand {\CP}{{\mathds P}}

\newcommand {\D}{{\mathcal D}}

\newcommand {\J}{{ J}}

\newcommand{\MM}{{\mathcal {MM}}}

\newcommand{\w}{{\omega}}
\newcommand{\ngam}{{n_{\epsilon}(\gamma)}}

\def\Pic{\operatorname{Pic}}
\def\Ker{\operatorname{Ker}}

\def\Supp{\operatorname{Supp}}

\def\div{\operatorname{div}}
\def\deg{\operatorname{deg}}
\def\mod{\operatorname{mod}}

\def\Im{\operatorname{Im}}
\def\ord{\operatorname{ord}}
\def\reg{\operatorname{reg}}
\def\Hom{\operatorname{Hom}}
\def\Ext{\operatorname{Ext}}

\title{The Fundamental Group and Extensions of Motives of  Jacobians of Curves}
\author{Subham Sarkar and Ramesh Sreekantan
}

\begin{document}

\baselineskip=17pt
\maketitle

\begin{abstract}
In this paper we construct extensions of mixed Hodge structure coming from the mixed Hodge structure on the graded quotients of the group ring of the Fundamental group of a smooth projective pointed curve which correspond to the regulators of certain motivic cohomology cycles on the Jacobian of the curve essentially  constructed by Bloch and Beilinson. This leads to a new iterated integral expression for the regulator.  This is a generalisation of a theorem  of Colombo \cite{colo} where she constructed the extension corresponding to Collino's cycles in the Jacobian of a hyperelliptic curve. 


\end{abstract}

\tableofcontents
\section{Introduction}

A formula, usually called Beilinson's formula --- though independently due to  Deligne as well --- describes the motivic cohomology group of a smooth projective variety $X$  over a number field as the group of extensions in a conjectured category of mixed motives, $\MM_{\Q}$. If $i$ and $n$ are two integers then \cite{scho2},  
$$\Ext^1_{\MM_{\Q}}(\Q(-n),h^i(X))=\begin{cases} CH^n_{\hom}(X) \otimes \Q & \text{ if } i+1=2n\\ H^{i+1}_{\M}(X,\Q(n)) & \text{ if } i+1\neq 2n\end{cases}$$
Hence, if one had a way of constructing extensions in the category of mixed motives by some other method,  it would provide a way of constructing motivic cycles. 

One way of doing so is by considering the group ring of the fundamental group of the algebraic variety $\ZZ[\pi_1(X,P)]$. If $J_P$ is the augmentation ideal --- the kernel of the map from $\ZZ[\pi_1(X,P)]   \to \ZZ$ ---  then the graded pieces $J_P^a/J_P^b$ with $a<b$ are expected to have a motivic structure. These give rise to natural extensions of motives --- so one could hope that these extensions could be used to construct natural motivic cycles. 

Understanding the motivic structure on the fundamental group is complicated.  However, the Hodge structure on the fundamental group is well understood \cite{hain}.  The regulator of a motivic cohomology cycle can be thought of as the realisation of the  corresponding extension of motives as an extension in the category of mixed Hodge structures. So while we may not be able to construct motivic cycles as extensions of {\em motives} coming from the fundamental group - we can hope to construct their regulators as extensions of {\em mixed Hodge structures} (MHS)  coming from the fundamental group. 

The aim of this paper is to describe this  construction in the case of  the motivic cohomology group of the Jacobian of a curve. The first work in this direction is due to Harris \cite{harr} and Pulte \cite{pult}, \cite{hain}. They  showed that the Abel-Jacobi image of the modified diagonal cycle on the triple product of a pointed curve $(C,P)$, or alternatively the Ceresa cycle in the Jacobian  $J(C)$ of the curve, is the same as an extension class coming from $J_P/J_P^3$, where $J_P$ is the augmentation ideal in the group ring of the fundamental group of $C$ based at $P$.  

In \cite{colo}, Colombo extended this theorem to show that the regulator of a cycle in the motivic cohomology of a Jacobian of a {\em hyperelliptic} curve, discovered by Collino \cite{coll},  can be realised as an extension class coming from $J_P/J_P^4$, where here $J_P$ is the augmentation ideal of a  related curve.

In this paper we extend Colombo's result to more general curves. If $C$ is a smooth projective  curve of genus $g$  with a function $f$ with divisor $\div(f)=NQ-NR$ for some points $Q$ and $R$ and some integer $N$  and such that $f(P)=1$ for some other point $P$,  there is a motivic cohomology  cycle  $Z_{QR,P}$ in $H^{2g-1}_{\M}(J(C),\ZZ(g))$ discovered by Bloch \cite{blocirvine}.  We show that the regulator of this cycle can be expressed in terms of an extensions coming from $J_P/J_P^4$. When $C$ is hyperelliptic and $Q$ and $R$ are ramification points of the canonical map to $\CP^1$, this is Colombo's result.

A crucial step in Colombo's work is to use the fact that  the modified diagonal cycle is {\em torsion} in the Chow group $CH^2_{\hom}(C^3)$ when $C$ is a   hyperelliptic curve. This means the  extension coming from $J_P/J_P^3$ splits and hence does not depend on the base point $P$. This allows her to consider the extension for $J_P/J_P^4$. In general, that is {\em not} true --- in fact the known examples of non-torsion modified diagonal cycles come from the curves we consider - namely modular and Fermat curves. Our main contribution is to use an idea of Rabi \cite{rabi}  to show that Colombo's  arguments can be extended to work in our  case as well. As a result we have a more general situation --- which has some arithmetical applications.  

In Colombo's paper there were errors in Propositions 3.2 and 3.3 which were pointed out by a referee of an earlier version of  this paper. Hence we had to make some revisions. As it turned out the statement of the main result still holds under some restricted conditions and much of the revisions we made were to understand these conditions. Unfortunately this has made the paper a little long. 

We have the following theorem ({\bf Theorem \eqref{mainthm}}):

 \begin{thm} Let $C$ be a smooth projective  curve of genus $g_C$ over $\C$.  Let  $P$, $Q$ and $R$ be three  distinct  points such that there is a function $f_{QR}$ with $\div(f_{QR})=NQ-NR$ for some $N$ and $f_{QR}(P)=1$.  Let    $Z_{QR}=Z_{QR,P}$ be the element of the motivic cohomology group $H^{2g-1}_{\M}(J(C), \ZZ(g))$ constructed below in Section \ref{cycleconstruction}. There exists an extension class  $\epsilon^4_{QR,P}$  in $\Ext^1_{MHS}(\ZZ(-2),\wedge^2 H^1(C))$ constructed from the mixed Hodge structures associated to  the fundamental groups $\pi_1(C\backslash Q,P)$ and $\pi_1(C \backslash R,P)$ such that 
 $$\epsilon^4_{QR,P}=(2g_C+1)N \reg_{\ZZ}(Z_{QR}) $$
 in $\Ext^1_{MHS}(\ZZ(-2),\wedge^2 H^1(C ))$.  
 
 \end{thm}
 
In other words our theorem states that the regulator of a natural cycle in the motivic cohomology group of the Jacobian of curves, 
being thought of as an extension class, is the same as  the extension class of a  natural extension of mixed Hodge structures 
coming from the fundamental group of the curve.   

The outline of the proof is as follows. In the second chapter we introduce background material on iterated integrals, regulators, motivic cycles and extensions. We then describe the cycle $Z_{QR,P}$ and give a formula for its regulator. In the next chapter we construct an extension coming  from the MHS on the fundamental group that we expect to be related to the regulator of  our explicit cycle. Finally, we prove that that is indeed the case. This is the place where there was an error in the paper of Colombo \cite{colo}. One has to restrict to a certain subspace of $H_1(C \setminus \{Q,R\})$ which is isomorphic to $H_1(C)$. There is no canonical way in which to choose this in general, but in our case the function $f_{QR}$ determines such as space.  

As a result of this we have a new  iterated  integral  formula for the regulator which is more amenable to computation. In a subsequent paper we apply this in the case of  Fermat curves to get an explicit expression for the regulator in terms of hypergeometric functions --- analogous to the works of Otsubo \cite{otsu1},\cite{otsu2}. As the primary requirement is a nice basis for the homology and the cohomology of the curve, we expect this will also work in the case of modular curve, though that case is quite well understood. 

Darmon-Rotger-Sols  \cite{DRS} have used the modified diagonal cycle to construct points on Jacobians of the curves and used the iterated integral approach to find a formula for the Abel-Jacobi image of these points. Starting with Bloch \cite{blocirvine} and later Collino \cite{coll} and Colombo \cite{colo} it has been known that these null homologous cycles degenerate to higher Chow cycles on related varieties. Recently Iyer and M\"uller-Stach \cite{iymu} have shown that the modified diagonal cycle degenerates to the kind of cycles we consider in some special cases. This degeneration can be understood from the point of view of extensions and in fact the iterated integral expression we have for the regulator show that in a rather natural way in terms of the holomorphic forms degenerating to logarithmic forms.

{\em Acknowledgements:} This work constitutes part of the PhD. thesis of the first author. We would like to thank  Najmuddin Fakhruddin, Noriyuki Otsubo, Satoshi  Kondo, Elisabetta  Colombo,  Jishnu Biswas,  Manish Kumar, Ronnie Sebastian, Ranier Kaenders, Harish Seshadri, Arvind Nair, Shreedhar Inamdar  and Suresh Nayak for their comments and suggestions. We would also like to thank the referee of an earlier version of this manuscript for pointing out numerous errors and specifically for pointing out the error in Colombo's paper which we had to rectify and the referee of this version for their comments and suggestions. 

Finally, it gives us great  pleasure to thank the Indian Statistical Institute, Bangalore  for their support while this work was done. The first author would also like to thank TIFR for their hospitality while these revisions were being made.

\section{Iterated Integrals, Cycles, Extensions and Regulators}

\subsection{Iterated Integrals}

Let $X$ be a smooth projective variety over $\C$. Let $\alpha:[0,1] \to X$ be a path and $\omega_1,\omega_2,\dots,\omega_n$ be $1$-forms on $X$. Suppose $\alpha^*(\omega_i)=f_i(t) dt$. The  {\em iterated integral of length n} of $\w_1,\w_2 \dots \w_n$ is defined to be 
$$\int_{\alpha} \omega_1\omega_2\dots \omega_n:=\int_{0 \leq t_1 \leq t_2\leq \dots \leq t_n \leq 1} f_1(t_1)f_2(t_2)\dots f_n(t_n)dt_1dt_2 \dots  dt_n.$$
An {\em iterated integral of length $\leq n$} is a linear combination of integrals of the form above with lengths $\leq n$. It is said to be a {\em homotopy functional} if it only depends on the homotopy class of the path $\alpha$.  A homotopy functional gives a functional on the group ring of the fundamental group or path space. 

Iterated integrals can be thought of as integrals on simplices and satisfy the following basic properties. Here we have only stated the results for length two iterated integrals, since that is the only type we will encounter in this paper.

\begin{lem}[Basic Properties] Let $\omega_1$ and $\omega_2$ be smooth $1$-forms on $X$ and $\alpha$ and $\beta$ piecewise smooth paths on $X$ with $\alpha(1)=\beta(0)$. Then
\begin{enumerate} 
\item $\displaystyle{\int_{\alpha \cdot \beta} \omega_1 \omega_2=\int_{\alpha} \omega_1\omega_2 + \int_{\beta} \omega_1 \omega_2 + \int_{\alpha} \omega_1 \int_{\beta} \omega_2}$
\item $\displaystyle{\int_{\alpha} \omega_1 \omega_2 + \int_{\alpha} \omega_2 \omega_1 =  \int_{\alpha} \omega_1 \int_{\alpha} \omega_2}$
\item $ \displaystyle{\int_{\alpha}  df \omega_1 = \int_{\alpha} f \omega_1 - f(\alpha(0)) \int_{\alpha} \omega_1}$
\item $\displaystyle{\int_{\alpha} \omega_1 df = f(\alpha(1))\int_{\alpha} \omega_1 -\int_{\alpha} f\omega_1}$
\end{enumerate}
\label{basicproperties}

\end{lem}

\begin{proof} This can be found in any article on iterated integrals, for instance Hain's excellent article \cite{hain}.\end{proof}

\subsection{Motivic Cohomology Cycles}
\label{motiviccycles}

Let $X$ be a smooth projective algebraic variety of dimension $g$  defined over $\C$. The motivic cohomology group $H^{2g-1}_{\M}(X,\ZZ(g))$  has the following presentation: Generators  are represented by finite sums 
$$Z=\sum_i (C_i,f_i)$$
where $C_i$ are curves  on $X$ and $f_i:C_i \longrightarrow \CP^1$ are functions on them subject to the co-cycle condition 
$$\sum_i \div(f_i)=0.$$
Relations in this group are defined as follows. If $Y$ is a surface on $X$ and  $f$ and $g$ are functions on $Y$, one has the Steinberg element  $\{f,g\}$ in $K_2(\C(Y))$, where $\C(Y)$ is the function field of $Y$. To such an element one can consider the sum, called the tame symbol of $\{f,g\}$, 
$$\tau(\{f,g\})=\sum_{W\in Y^{(1)}} (W,(-1)^{\ord_W(f)\ord_W(g)}\frac{f^{\ord_W(g)}}{g^{\ord_W(f)}}).$$
where $Y^{(1)}$ is the  collection of curves on $Y$. This is a finite sum and satisfies the co-cycle condition, hence lies in the above group.  An element is said to be $0$ in $H^{2g-1}_{\M}(X,\ZZ(g))$  if it lies in the image of the free abelian group generated by the tame symbols of elements of $K_2(\C(Y))$ for some surface $Y \subset X$. The group $H^{2g-1}_{\M}(X,\ZZ(g)) \otimes \Q$ is the same as the higher Chow group $CH^g(X,1) \otimes \Q$.

In the group $H^{2g-1}_{\M}(X,\ZZ(g))$ there are certain {\em decomposable} cycles coming from the product 
$$H^{2g-1}_{\M}(X,\ZZ(g))_{dec}= \Im \left( H^{2g-2}_{\M}(X,\ZZ(g-1) \otimes  H^1_{\M}(X,\ZZ(1))   ) \longrightarrow H^{2g-1}_{\M}(X,\ZZ(g))\right).$$ 
This is simply $CH^1(X) \times \C^*$. The group of  {\em indecomposable} cycles  is defined as the quotient 
$$H^{2g-1}_{\M}(X,\ZZ(g))_{ind}=H^{2g-1}_{\M}(X,\ZZ(g))/H^{2g-1}_{\M}(X,\ZZ(g))_{dec}.$$
In general it is not easy to find non trivial  elements in this group.

\subsection{The Cycle $Z_{QR,P}$ on  on $J(C)$}
\label{cycleconstruction}

In this section we construct a motivic cohomology cycle on $J(C)$, where $C$ is a smooth projective curve over $\C$. This was  first constructed by Bloch \cite{blocirvine} in the case when $C$ is the modular curve $X_0(37)$. The cycle is similar, in fact, generalises, the cycle constructed by Collino \cite{coll}. This section generalises the work of Colombo \cite{colo} on constructing the extension corresponding to the Collino cycle and hence many of the arguments are adapted from her paper. 

Let $C$ be a smooth projective curve defined over $\C$. Let $Q$ and $R$ be two distinct  points on $C$ such that there is a function  $f=f_{QR}$ with divisor 
$$\div(f_{QR})=NQ-NR$$
for some $N \in \NN$.  To determine the function precisely, we choose a distinct third point $P$ and assume $f_{QR}(P)=1$. 

There exist notable examples of curves where such functions can easily be found. For instance, {\em modular curves} with $Q$ and $R$ being cusps, {\em Fermat curves} with the two points being among  the `trivial' solutions of Fermat's Last Theorem, namely the points with one of the coordinates being $0$, and  {\em hyperelliptic curves} with the two points being Weierstrass points. 

Let $C^Q$ denote the image of $C$ under the map $C \rightarrow J(C)$ given by $x \rightarrow x-Q$. Similarly, let $_RC$ denote the image of $C$ under the map $x \rightarrow R-x$ and let $f^Q$ and $_Rf$ denote the function $f$ being considered as a function on $C^Q$ and $_RC$ respectively. 

Consider the cycle in $J(C)$ given by 
$$Z_{QR,P}=(C^Q,f^Q)+(_RC,_Rf)$$ 
we have 
$$\div_{C^Q} (f^Q) + \div_{_RC}(_Rf)=N(0)-N(R-Q) + N(R-Q)-N(0)=0.$$
Hence the cycle $Z_{QR,P}$ gives an element of $H^{2g-1}_{\M}(J(C),\ZZ(g))$.

\subsection{Regulators}

Let $X$ be a smooth projective  variety of dimension $g$  over $\C$. The regulator map of Beilinson is a  map from the motivic cohomology group to the Deligne cohomology group.
$$\reg_{\ZZ}: H^{2g-1}_{\M}(X, \ZZ(g)) \rightarrow H^{2g-1}_{\D}(X, \ZZ(g))= \frac {(F^1H^2(X,\C))^*}{H_2(X,\ZZ(1))}.$$
where $*$ denotes the $\C$-linear dual and $F^{\bullet}$ denotes the Hodge filtration. The group $H^{2g-1}_{\D}(X,\ZZ(g))$ is a generalised torus. 

The map is defined as follows \cite{coll}: Let   $Z=\sum_i (C_i,f_i)$ be a cycle   in $H^{2g-1}_{\M} (X,\ZZ(g))$, so $C_i$ are curves on $X$ and $f_i$ are functions on them satisfying the cocycle condition. Let $[0,\infty]$ denote the positive real axis in $\CP^1$ and $\gamma_i=f_i^{-1}([0,\infty])$. Then $\sum_i \div(f_i)=0$ implies that the $1$-chain $\sum_i \gamma_i$ is closed and  in fact torsion.  If $H_1(X,\ZZ)$ has no torsion -- as is in the case of the Jacobian of  curves, it is exact.  Assuming that, we have  
$$\sum_i \gamma_i=\partial(D)$$
for some  $2$-chain $D$. For   a closed $2$-form  $\omega$ whose cohomology class lies in $F^1H_{DR}^2(X,\C)$,
\begin{equation}
\reg_{\ZZ}(Z)(\omega):=\sum_i \int_{C_i \setminus \gamma_i } \log(f_i) \omega + 2\pi i  \int_D \omega.
\label{regulatorformula}
\end{equation}
Here $C_i\setminus \gamma_i$ is the Riemann surface with boundary obtained as follows. Let $n_{\epsilon}(\gamma_i)$ be an open tubular neighbourhood of $\gamma_i$ in $C_i$ which is homeomorphic to $(-\epsilon,\epsilon) \times \gamma_i$. $C_i\backslash n_{\epsilon}(\gamma_i)$ is a closed subset of $C_i$ with the structure of a manifold with boundary. The boundary $\partial(C_i \setminus n_{\epsilon}(\gamma_i))$ is made up of two copies of $\gamma_i$ with opposite orientation as well as $(-\epsilon,\epsilon) \times \gamma_i(0)$ 
and $(-\epsilon,\epsilon) \times \gamma_i(1)$ with opposite orientations. $C_i \setminus \gamma_i$ is the manifold obtained by letting $\epsilon \rightarrow  0$.  
For a decomposable element $(C,a)$, where $a \in \C^*$,  the regulator is particularly simple:
$$\reg_{\ZZ}((C,a))(\omega)=\int_{C} \log(a)\omega=\log(a) \int_C \omega.$$

\subsection{The Regulator of $Z_{QR,P}$} 

Let $Z_{QR,P}$ be the motivic cohomology cycle in $H^{2g-1}_{\M}(J(C), \ZZ(g))$.  We now obtain a formula for its regulator.  The regulator is a current on forms in $F^1(H^2(J(C)_{\C}))$. Since $H^2(J(C))=\wedge^2 H^1(C)$ elements are of the form $\phi \wedge \psi$ where $\phi$ and $\psi$ are closed $1$-forms on $C$ and one of $\phi$ or $\psi$ is of type $(1,0)$. We have the following theorem: 
 \begin{thm} Let $Z_{QR,P}$ be the motivic cohomology cycle in  $H^{2g-1}_{\M}(J(C),\ZZ(g))$,  $\phi$ and  $\psi$ two closed harmonic  $1$-forms in  $H^1(J(C))=H^1(C)$ with $\psi$ holomorphic. Then 
 $$
\reg_{\ZZ}(Z_{QR,P})( \phi \wedge \psi) = 2 \int_{C\setminus \gamma} \log(f)\phi \wedge \psi   + 2\pi i  \int_{\gamma} (\phi\psi -  \psi\phi )
\label{regform}
$$
\end{thm}

\begin{proof}
Recall that $f=f_{QR}$ is a function on $C$ with divisor $NQ-NR$ for some $N$. Let  $\omega=\phi \wedge \psi$ and $\gamma=f^{-1}([0,\infty])$. As $f$ is of degree $N$, $\gamma$ is the union of $N$ paths each lying on a different sheet with only the points $Q$ and $R$ in common. We will denote them by $\gamma^i$, $1\leq i \leq N$. Each $\gamma^i$ is a path from $Q$ to $R$.  Let $\gamma^Q$ and $_R\gamma$ denote the path $\gamma$ on $C^Q$, $_RC$ respectively and similarly for the components $\gamma^i$. Then from the co-cycle condition one has 
$$  \gamma^Q \cdot _R\gamma=\partial(D)$$
where $D$ is a $2$-chain on $J(C)$. 

From equation \eqref{regulatorformula}  one has 
\begin{equation}
\reg_{\ZZ}(Z_{QR,P})(\omega )=\int_{C^Q \setminus \gamma^Q} \log(f_Q ) \omega + \int_{_RC \setminus _R\gamma} \log(_Rf) \omega+ 2\pi i \int_{D} \omega.
\label{regfor1}
\end{equation}
Our aim is to find a more explicit expression for $\reg_{\ZZ}(Z_{QR,P})$. For this we need an explicit description of $D$. This was done by Colombo \cite{colo}, Lemma 1.2.
\begin{lem} Let 
\[
 a(s,t)=t \hspace{2cm} {\rm and }  \hspace{2cm}  b(s,t)=\frac{t(1-s)}{1-s(1-t)}. \]
Define $F_i:[0,1] \times [0,1] \longrightarrow J(C)$ by 
$$F_i(s,t)=\gamma^i(a(s,t)) - \gamma^i(b(s,t))$$
for $1 \leq i \leq N$ and let 
$$D_i=\Im(F_i).$$
Then, orienting counterclockwise, 
\[
\partial(D_i)= \gamma^{i,Q} \cdot  _R\gamma^{i}.
\]
In particular,  if $D=\cup_{i=1}^{N} D_i$ then 
\[
\partial(D)=  \gamma^Q \cdot _R\gamma
\]
\end{lem}

\begin{proof} The proof is essentially identical to Colombo's Lemma 1.2 - the only change is that she does it for $N=2$ - so we do not repeat it here.

\end{proof}
We can compute the last integral as an iterated integral as follows. 

\begin{lem} Let $\phi$ and $\psi$ be closed harmonic $1$-forms on $C$  and let $D_i$ be a disc as   in the above lemma. Then 
$$\int_{D_i} \phi \wedge \psi = \int _{\gamma^{i,Q}}  \phi \psi - \int_{{_R\gamma}^i} \psi \phi=\int_{\gamma} (\phi \psi - \psi \phi)$$
\label{secondintegral}
\end{lem}

\begin{proof} This again is a slightly modified version of Colombo \cite{colo}, Lemma 13.  
 \end{proof}
 
 This completes the proof of Theorem \ref{regform}.

\section{Extensions}

As stated in the introduction, conjecturally, there is a canonical  description of the motivic cohomology group as an extension in the category of mixed motives. From now on, $\Ext$ will denote $\Ext^1$. Further, we will use $H^*(X)$ to denote the group $H^*(X(\C),\ZZ)$, the singular (Betti) cohomology group with integral coefficients  and $H^*(X)_A$ to denote $H^*(X)\otimes_{\ZZ} A$,  where $A$ is typically $\Q, \R$ or $\C$. 

In our case, if one has a suitable category of mixed motives over $\Q$,  $\MM_{\Q}$,  one expects  for a variety $X$ \cite{scho2} and $i,n$ non-negative numbers  with $i<2n-1$, 
\begin{equation}
H^{i+1}_{\M} (X,\Q(n))\simeq \Ext_{\MM_{\Q}} (\Q(-n),h^i(X)). 
\label{extmot}
\end{equation}
where $\Q(-n)$ denotes the twist of the Tate motive and $h^i(X)$ denotes the motive whose Hodge realisation is $H^i(X)$. 

One {\em knows} that the Deligne cohomology can be considered as an extension in the category of integral mixed Hodge structures, 
$$H^{i+1}_{\D} (X,\ZZ(n))\simeq \Ext_{MHS} (\ZZ(-n),H^i(X)) $$

Assuming \eqref{extmot} holds at the level of integer coefficients,  the regulator map  above then has a canonical description as the map induced by the realisation map from the category of mixed motives to the  category of mixed Hodge structures, 
$$\Ext_{\MM_{\Q}} (\ZZ(-n),h^i(X)) \stackrel{\reg_{\ZZ}}{\longrightarrow} \Ext_{MHS}(\ZZ(-n),H^i(X)).$$

\subsection{Extensions of Mixed Hodge structures coming from the Fundamental Group}

The key point of this paper is that,  in some cases, one can also obtain extensions of mixed Hodge structures in other ways. For instance, if $(X,P)$ is a pointed algebraic variety, it was shown by Hain \cite{hain}  that the graded quotients $J_P^a/J_P^b$, with $a\leq b$,  where $J_P$ is the augmentation ideal of the  group ring of the fundamental group $\ZZ[\pi_1(X,P)]$, carry mixed Hodge structures. Hence natural exact sequences involving them lead to extensions of mixed Hodge structures. 

Our aim is to first construct some natural motivic cohomology cycles in the case when $X=J(C)$, the Jacobian of a curve of genus $g=g_C$. Their regulators will give rise to extensions of mixed Hodge structures.  We will show that there are natural extensions of mixed Hodge structures coming from the Hodge structure on the graded pieces of  $\ZZ[\pi_1(C,P)]$  for some suitable point $P$  which give the {\em same}  extensions.  In particular, since the constructions can be carried out in at the level of mixed motives, if we had a good category of mixed motives  the  cycle {\em itself} would be an extension in the conjectured category of mixed motives coming from the fundamental group.

\subsection{The Mixed Hodge structure on the Fundamental Group} 

Let $C$ be a smooth projective  curve over $\C$  and $P$, $Q$ and $R$ be three distinct  points on $C$. 
Consider the open curve $C_Q=C -\{Q\}$. Let $\ZZ[\pi_1(C_Q,P)]$ be the group ring of the fundamental group of $C_Q$ based at $P$.  Let $J_{Q,P}:=J_{C_Q,P}$ denote the augmentation ideal 
$$J_{Q,P}:=J_{C_Q,P}=\Ker \{ \ZZ[\pi_1(C_Q,P)] \stackrel{\deg}{\longrightarrow} \ZZ \}.$$
Let $H^0({\mathcal B}_r(C_Q;P))$ denote the $F$-vector space, where $F$ is $\R$ or $\C$, of homotopy invariant iterated integrals of length $\leq r$.  Chen\cite{chen} showed that    
$$H^0({\mathcal B}_r(C_Q;P)) \simeq  \Hom_{\ZZ}(\ZZ[\pi_1(C_Q,P)]/J_{C_Q.P}^{r+1},F)$$
under the map 
$$ I \longrightarrow   I(\gamma)=\int_{\gamma} I .$$
Using this Hain \cite{hain} was able to put  a natural  mixed Hodge structure on the graded pieces $J_{Q,P}/J_{Q,P}^r$. 

\subsection{The Extension $E^3_{Q,P}$}

From this point on, we will use the following notation. For an extension $E$ of mixed Hodge structures, 
$$E:0 \longrightarrow  B \longrightarrow H \longrightarrow A \longrightarrow 0$$
we use $m$ to denote its class in $\Ext^1_{MHS}(A,B)$ and $H$ to denote the middle term. We will also use the notation $N\cdot E$ to denote $N$ times the extension with respect to the Baer sum, 
use $N \cdot m$ to denote its class of this extension in the $\Ext$ group and $N\cdot H$ to denote its middle term .

For $r\geq3$, one can consider the extensions of mixed Hodge structures 
$$E^r_{Q,P}: 0 \longrightarrow (J_{Q,P}/J_{Q,P}^{r-1})^* \longrightarrow (J_{Q,P}/J_{Q,P}^{r})^*\longrightarrow (J_{Q,P}^{r-1}/J_{Q,P}^{r})^*\longrightarrow 0$$
where for a module $M$, $M^*=\Hom(M,\ZZ)$. 

The simplest non-trivial case is when $r=3$.  In this case $(J_{Q,P}/J_{Q,P}^2)^* \simeq H^1(C_{Q})\simeq H^1(C)$ and  $(J_{Q,P}^2/J_{Q,P}^3)^* \simeq  \otimes^2 H^1(C)$ and the exact sequence becomes 
$$E^3_{Q,P}: 0 \longrightarrow H^1(C) \longrightarrow (J_{Q,P}/J_{Q,P}^3)^* \longrightarrow \otimes^2 H^1(C) \longrightarrow 0.$$
Hence $E^3_{Q,P}$ gives an element $m^3_{Q,P}$  in $\Ext(\otimes^2 H^1(C),H^1(C))$. A similar construction with $R$ in the place of $Q$ gives us the extension $E^3_{R,P}$, which also lies in the same $\Ext$ group.

There is  a surjection $\cup:\otimes^2 H^1(C) \longrightarrow  H^2(C)  \simeq \ZZ(-1)$ coming from the cup product. Let $K$ be the kernel of this map. The exact sequence  of Hodge structures 
$$0 \longrightarrow K \longrightarrow \otimes^2 H^1(C) \stackrel{\cup}{\longrightarrow} \ZZ(-1)  \longrightarrow 0$$
splits over $\Q$  but {\em not} over $\ZZ$.  This happens as follows: There is a bilinear form \cite{kaen}
$$b:\otimes^2 H^1(C) \times \otimes^2 H^1(C) \longrightarrow \ZZ$$
defined by 
$$b(x_1 \otimes x_2,y_1 \otimes y_2)=(x_1 \cup y_2) \cdot (x_2 \cup y_1).$$
Let $S$ denote the orthogonal complement of $K$ in $\otimes^2 H^1(C)$ with respect to this bilinear form. Then, under the cup product  $S$ projects to  $2g_C \ZZ(-1)$ where $g_C$ is the genus of $C$  and 
$$\otimes^2 H^1(C)_{\Q}=K_{\Q} \oplus S_{\Q}.$$
Let $\bar{m}^3_{Q,P}$ denote the class  in $\Ext_{MHS}(S,H^1(C))$ corresponding to the extension 
$$0 \longrightarrow H^1(C) \longrightarrow \bar{E}^3_{Q,P} \longrightarrow S \longrightarrow 0$$
obtained by restricting $E^3_{Q,P}$ to the extension of $S$ by $H^1(C)$. From Kaenders \cite{kaen} one knows there is a covering map of complex tori, 
$$\Ext(\otimes^2 H^1(C),H^1(C)) \stackrel{\phi}{\longrightarrow} \Ext(K \oplus S, H^1(C))=\Ext(K,H^1(C)) \times \Ext(S,H^1(C)).$$
It is well known that $\Ext(S,H^1(C))=\Ext(\ZZ(-1),H^1(C))\simeq \Pic^0(C)$.  To understand the other term, from the work of
 Hain \cite{hain}, Pulte \cite{pult}, Kaenders \cite{kaen} and Rabi \cite{rabi} one has the following theorem 
\begin{thm} The image of the  class  $m^3_{Q,P}$ of $E^3_{Q,P}$ in $\Ext(\otimes^2 H^1(C) ,H^1(C))$  is given by 
 $$\phi(m^3_{Q,P})=(m^3_P, \bar{m}^3_{Q,P})$$ 
 where $m^3_P \in \Ext(K,H^1(C))$ depends only on $P$ and $\bar{m}^3_{Q,P}$ is given by 
 
$$2 g_C Q  - 2P -\kappa_C \in \Pic^0(C)$$
where $\kappa_C$ is the canonical divisor of  $C$ and $g_C$ is the genus of $C$. 
\label{moddiag}
\end{thm}

Recall that in the group $\Ext$,  addition is given by the Baer sum. We will denote this by $\oplus_B$  (or $\ominus_B$ if we are taking differences). Let $m^3_{QR,P}$ denote the Baer difference $m^3_{Q,P} \ominus_{B} m^3_{R,P}$.
\begin{lem} Under the hypothesis that there is a function with divisor $\div(f_{QR})=NQ-NR$  the extension class $m^3_{QR,P}$ is torsion in $\Ext(H^1C) \otimes H^1(C),H^1(C))$.  Precisely, 
$$N\cdot H^3_{QR,P} \simeq H^1(C) \bigoplus \otimes^2 H^1(C)$$
where by $N\cdot H^3_{QR,P}$ we mean the middle term of the exact sequence obtained by adding the sequence $E^3_{QR,P}$  to itself $N$-times using the Baer sum. 
\label{splitting}
\end{lem}

\begin{proof} This follows from \cite{kaen}[Theorem 2.5] which states that the map 
$$\Pic^0(C) \longrightarrow \Ext( H^1(C) \otimes H^1(C),H^1(C))$$ 
given by 
$$Q-R \longrightarrow m^3_{Q,P}-m^3_{R,P}$$
is well defined and  injective. Hence, since $N(Q-R)=0$ in $\Pic^0(C)$, $N(m^3_{Q,P}-m^3_{R,P})=N(m^3_{QR,P})=0$ in $\Ext( H^1(C) \otimes H^1(C),H^1(C))$. 
\end{proof}

A consequence of this is that there is a morphism of integral mixed Hodge structures 
$$r_{3}: N \cdot H^3_{QR,P} \longrightarrow  H^1(C)$$
given by the projection.

\begin{rem}
This extension represents the class $Q-R$, at least up to a integral multiple, and is hence the first example of the theme of this paper - namely the Abel-Jacobi image of a  null-homologous cycle is described in terms of  extensions coming from the fundamental group. 
\end{rem}

\end{proof}

\subsection{The Extensions $E^4_{Q,P}$ and $E^4_{R,P}$}

From the work of Hain, Pulte, Harris and others one knows that the class $m^3_P$ in $\Ext(K,H^1(C))$ corresponds to the extension of mixed Hodge structures determined by the Ceresa cycle in $J(C)$, or alternately,  the modified diagonal cycle in $C^3$. We would like to construct a similar class corresponding to the motivic cohomology cycle $Z_{QR,P}$.  To that end,  we now consider, with $C, P,Q$ and $R$ as before, the extension corresponding to $r=4$
$$E_{Q,P}^4: 0 \longrightarrow (J_{Q,P}/J_{Q,P}^3)^*  \longrightarrow (J_{Q,P}/J_{Q,P}^4)^* \longrightarrow (J_{Q,P}^3/J_{Q,P}^4)^* \longrightarrow 0.$$
We have that $(J_{Q,P}^3/J_{Q,P}^4)^*\simeq \otimes^3 H^1(C)$ and this does not depend on $P, Q$ or $R$. However, from Theorem \ref{moddiag},  $(J_{Q.P}/J_{Q,P}^3)^*$ depends on $Q$ and $P$. Similarly  $(J_{R.P}/J_{R,P}^3)^*$ depends on $R$ and $P$.  Hence we get classes in $\Ext(\otimes^3 H^1(C),(J_{Q.P}/J_{Q,P}^3)^*)$ and $\Ext(\otimes^3 H^1(C),(J_{R.P}/J_{R,P}^3)^*)$ -- which are different groups - hence we cannot take their difference. 

When $C$ is hyperelliptic  the extension classes $m^3_{Q,P}$ and $m^3_{R,P}$  are 2-torsion in  $\Ext(\otimes^2 H^1(C),H^1(C))$. Hence one gets two classes 
$$2m^4_{Q,P},  2m^4_{R,P}  \in \Ext(\otimes^3 H^1(C),\otimes^2 H^1(C) \oplus H^1(C))$$
and one can project to get two classes $e^4_{Q,P}$ and $e^4_{R,P}$ in $\Ext(\otimes^3 H^1(C), H^1(C))$.  Colombo \cite{colo} shows that the class 
$$e^4_{QR,P}=e^4_{Q,P} \ominus_B e^4_{R,P} \in \Ext(\otimes^3 H^1(C), H^1(C))$$
corresponds to the extension determined by  the  cycle  $Z_{QR}$ --- after pulling back and pushing forward with some standard maps.

Unfortunately, in general the extension classes  $m^3_{Q,P}$ and $m^3_{R,P}$ are {\em not} torsion in the $\Ext$ group. They correspond to the instances where the Ceresa cycle is non-torsion --- which is the generic case. In fact,  the instances where it is  known that the cycles are non-torsion are precisely the cases we have in mind --- modular curves and Fermat curves \cite{harr},\cite{blocspec}. Hence we cannot use this argument immediately. However, since we know from Lemma \ref{splitting} that their difference $m^3_{QR,P}$ is torsion, we would like to get an extension of the form  
$$0 \longrightarrow H^3_{QR,P}\otimes \Q \longrightarrow ``H^4_{QR,P}" \otimes \Q \longrightarrow \otimes^3 H^1(C)\otimes \Q \longrightarrow 0$$
where $``H^4_{QR,P}"$ is the middle term of a sort of  generalised Baer difference of the two extensions $E^4_{Q,P}$ and $E^4_{R,P}$. We could then  push-forward this extension using the splitting to get a class in $\Ext(\otimes^3 H^1(C)_{\Q},H^1(C)_{\Q})$. We cannot simply consider $E^4_{QR,P}=E^4_{Q,P} \ominus_{B} E^4_{R,P}$ as the two extensions lie in different $\Ext$ groups.   So we have to consider a generalisation of Baer sums to not necessarily exact sequences which we came across in a paper of Rabi \cite{rabi}.

\subsection{The Baer sum}

This is well known but we recall it to fix notation in order to describe Rabi's work. Recall that if  we have two exact sequences of modules 
$$\begin{CD}
E_j: 0 @>>> A @>f_j>> B_j @>p_j>> C @>>>0 
\end{CD}$$
for $j \in \{1,2\}$, the Baer difference  $E_1 \ominus_B E_2$  is constructed as follows. We have 
$$\begin{CD}
0 @>>> A \oplus A @>f_1\oplus f_2>> B_1 \oplus B_2 @>p_1\oplus p_2>> C\oplus C @>>> 0.
\end{CD}
$$  
Let  $\psi: B_1\oplus B_2 \longrightarrow C$ be the map 
$$\psi(b_1,b_2)=p_1(b_1)-p_2(b_2)$$
and let  $H=\Ker(\psi)=\{(b_1,b_2)| \; p_1(b_1)=p_2(b_2)\}$.  Let $D$ be the image of $\tilde{f}: A \longrightarrow A \oplus A  \longrightarrow H$
$$\tilde{f}(a)=(f_1(a),f_2(a))$$ 
Let $B=H/D$. The map $f: A \oplus A \longrightarrow B$ given by 
$$f(a_1,a_2)=(f_1(a_1),f_2(a_2))$$
factors through $(A \oplus A)/A \simeq A$ and so one has a map $\bar{f}:A \longrightarrow B$, 
$$a \longrightarrow (f_1(a),0)=(0,-f_2(a))$$ 
and an exact sequence 
$$\begin{CD}
0 @>>> A @>\bar{f}>> B@>p_1( or \;p_2) >> C @>>> 0
\end{CD}
$$
The class of this exact sequence  in $\Ext(C,A)$ is the Baer difference  $E_1 \ominus_B E_2$.  The Baer sum $E_1 \oplus_B E_2$  is the sequence obtained when one of the maps $f_2$ or $p_2$ is replaced by its negative.  The Baer sum is essentially the push-out over $A$ in the category of modules.

\subsection{Rabi's  generalisation} Now suppose we  have diagrams of the  following type:
\[
\begin{CD}
 &&0\\&&@VVV\\
 &&A_1\\
 &&@VVi_jV\\
 0 @>>> B_1^j @>f_j>> B_2^j @>p_j>> B_3 @>>>0  \\
&&@VV\pi_j V \\
&&C_1\\
&&@VVV\\
&&0
\end{CD}
\]
 where the vertical and horizontal sequences are exact for  $j\in \{1,2\}$.   Let $E_j$ denote the horizontal exact sequences:
$$\begin{CD}
E_j: 0 @>>> B_1^j @>f_j>> B_2^j @>p_j>> B_3 @>>>0.
\end{CD}$$
 We  would like to take the Baer  difference of the $E_j$ --- but since they do not lie in the same $\Ext$ group we cannot quite do that. However, we can still salvage something. 
 
One gets two types of extension classes in $\Ext$ groups which do not depend on $j$. The vertical exact sequences give  classes  in $\Ext(C_1,A_1)$. We can form their Baer difference to get an exact sequence 
$$
\begin{CD} 
0 @>>> A_1 @>>> \BB_1 @>>> C_1 @>>>0.
\end{CD}
$$
The horizontal exact sequences give extensions in $\Ext(B_3,B_1^j)$. These depend on $j$ but their push forward under  $\pi_j$ give classes ${\mathbf f}_{B_2^j}$ in  $\Ext(B_3,C_1)$.  

Define $\BB_2$ as follows: Let $H_2=\Ker(\psi)$, where $\psi$ is the `difference'  map 
 $$ \psi: B_2^1\oplus B_2^2 \longrightarrow B_3 $$
 $$ \psi((b_2^1,b_2^2)) = (p_1(b_2^1)-p_2(b_2^2))$$
 Let $D_2$ be the image of the map 
 $$A_1 \longrightarrow B_1^1 \oplus B_1^2 \longrightarrow H_2$$
 $$ a \longrightarrow (f_1(i_1(a)),f_2(i_2(a)))$$
Define $\BB_2=H_2/D_2$. We call this the {\em generalised} Baer difference of $E_1$ and $E_2$ and denote it by $\tilde{\ominus}_B$.  Observe that this is almost the Baer difference  of $E_1$ and $E_2$ in the sense that if $B_1=B_1^1=B_1^2$, then we could take the difference  in $\Ext(B_3,B_1)$. Since that is not the case, we do the best we can ---  we take the difference  of the {\em inexact} sequences 
 $$
 \begin{CD}
 0 @>>>A_1 @>>> B_2^j @>>>B_3 @>>> 0.
 \end{CD}
 $$
 As a result of this one has a complex
 $$
 \begin{CD}
 0 @>>> \BB_1 @>f_1 \oplus f_2 >> \BB_2 @>p_1( or \; p_2)>> B_3 @>>>0
 \end{CD}
 $$
 However, this complex is {\em not} exact since $\Ker(p_1)$ is  larger than $(f_1\oplus f_2)(\BB_1)$.  The next lemma describes this difference. 
 \begin{lem}[Rabi\cite{rabi}] Let $\BF=\BF_{B^1_2 \tilde{\ominus}_B B^2_2} = \BB_2/\BB_1$. Then one has the following diagram, in which the horizontal and vertical sequences are exact. 
 $$
 \begin{CD}
 &&&&&&0\\
 &&&&&&@VVV\\
 &&&&&&C_1\\
 &&&&&&@VV\phi V\\
 0 @>>> \BB_1 @>f>> \BB_2 @>\eta>> \BF @>>>0  \\
&&&&&&@VV\bar{p}V \\
&&&&&&B_3\\
&&&&&&@VVV\\
&&&&&&0
\end{CD}
 $$
 \label{rabilemma}
\end{lem}

\begin{proof} \cite{rabihodge}, Appendix B. We repeat the proof here as that is unpublished. The horizontal sequence is exact by definition. To show the vertical sequence is exact we have to first describe be map $\phi$.
It is defined as follows. One has maps $\pi_j:B_1^{j} \longrightarrow C_1$. Consider the natural map 
$$\begin{CD}
\tilde{\phi}:C_1\oplus C_1 @>>> (B_1^1 \oplus B_1^2 )/\Delta_{A_1}@>(f_1,f_2)>>\BB_2=H_2/D_2
\end{CD}
$$
$$\tilde{\phi}(c_1,c_2) \rightarrow (\pi_1^{-1}(c),\pi_2^{-1}(c)) \rightarrow (f_1(\pi_1^{-1}(c_1)),f_2(\pi_2^{-1}(c_2)))$$
where $\Delta_{A_1}=\{(i_1(a),i_2(a))| a \in A_1\}$. $\phi$  gives a  well defined map 
$$(C_1\oplus C_1)/\Delta_{C_1} \longrightarrow \BB_2/ \tilde{\phi}(\Delta_{C_1})$$
where  $\Delta_{C_1}=\{(c,-c) | c \in C_1\}$ is the anti-diagonal.  This is well defined as if $(b_1,b_2)$ and $(b'_1,b'_2)$ are in $(\pi_1^{-1}(c_1),\pi_2^{-1}(c_2))$ we have to show 
$$(f_1(b_1),f_2(b_2)) \equiv (f_1(b'_1),f_2(b'_2)) \mod \tilde{\phi}(\Delta_{C_1})$$
or 
$$(f_1(b_1-b'_1),f_2(b_2-b'_2)) \in \tilde{\phi}(\Delta_{C_1}).$$
From exactness, we have $b_1-b'_1=i_1(a_1)$ and $b_2-b'_2=i_2(a_2)$  with $a_i \in A_1$. The image of $\Delta_{C_1}$ under $(\pi_1^{-1},\pi_2^{-1})$ consists of $(b,b')$ such that $\pi_1(b)=\pi_2(b')$. $(i_1(a_1),i_2(a_2))$ lie in this image, hence
$$(f_1(i_1(a_1)),f_2(i_2(a_2))) = (f_1(b_1-b'_1),f_2(b_2-b'_2)) \in  \tilde{\phi}(\Delta_{C_1}).$$
Note that the pre-image  $(\pi_1^{-1},\pi_2^{-1})(\Delta_{C_1})$ in $B_1^1 \oplus B_1^2)/\Delta_{A_1}$  is the Baer difference $\BB_1$. Further, $(C_1\oplus C_1)/\Delta_{C_1}) \simeq C_1$. Hence one has a  map $\phi: C_1 \rightarrow \BF=\BB_2/\BB_1$ and  we get a exact sequence 
$$0 \longrightarrow  C_1 \stackrel{\phi}{\longrightarrow} \BF_{B^1_2 \tilde{\ominus}_B B^2_2}  \stackrel{\bar{p}}{\longrightarrow} B_3 \longrightarrow 0$$
This sequence is exact as if $b=(b^1_2,b^2_2)$ is in $\BF_{B^1_2 \tilde{\ominus}_B B^2_2}$ and $\bar{p}(b)=0$, then $p_1(b^1_2)=p_2(b^2_2)=0$. So $b^1_2$ and $b^2_2$ lie in the image of  $B_1^1\oplus B_1^2$ --- say $b_2^1=f_1(b_1^1)$ and $b_2^2=f_2(b_1^2)$.  Let $c_i=\pi_1(b_1^1)$ and $c_2=\pi_2(b_1^2)$.  Then 
$$b=\phi(c_1,c_2)$$
so it lies in the image of $\phi$.

\end{proof}

In general,  for any $\ZZ$-linear combination $m \cdot B^1_2 \tilde{\ominus}_B \;n \cdot B^2_2$ of $B_2^1$ and $B_2^2$ we get an extension class ${\mathbf f}_{m\cdot  B^1_2 \tilde{\ominus}_B \;n \cdot B^2_2}$  in $\Ext(B_3,C_1)$ corresponding to $\BF_{m \cdot B^1_2 \tilde{\ominus}_B \;n \cdot B^2_2}$.  The relation between this and the extension classes constructed above is given as follows:
\begin{cor} Let ${\mathbf f}_{B_2^j}$ and ${\mathbf f}_{m \cdot B_2^1 \tilde{\ominus}_B \; n \cdot B^2_2}$ be the extensions in $\Ext(B_3,C_1)$  described  above. Then, 
$${\mathbf f}_{m \cdot B_2^1 \tilde{\ominus}_B\; n \cdot B^2_2}=m\cdot {\mathbf f}_{B_2^1} \ominus_B \; n\cdot {\mathbf f}_{B_2^2}.$$
\label{newextensions}
\end{cor}

\begin{proof} This follows from the construction of the map $\phi$. 
\end{proof}
In the next section we apply these constructions in our particular case to get the extension class we want.

\subsection{The Extension $e^4_{QR,P}$}

In this section we construct an extension $e^4_{QR,P}$ in $\Ext(\otimes^3 H^1(C),H^1(C))$ which generalises the element $e^4_{Q,P} \ominus_B e^4_{R,P}$ constructed by Colombo. Recall that we have an exact sequence 
 $$E^3_{Q,P}: 0 \longrightarrow H^1(C) \longrightarrow (J_{Q,P}/J_{Q,P}^3)^* \longrightarrow \otimes^2 H^1(C) \longrightarrow 0$$ 
 and a similar one  $E^3_{R,P}$. Also, we have the  exact sequence 
 $$E_{Q,P}^4: 0 \longrightarrow (J_{Q,P}/J_{Q,P}^3)^*  \longrightarrow (J_{Q,P}/J_{Q,P}^4)^* \longrightarrow (J_{Q,P}^3/J_{Q,P}^4)^* \longrightarrow 0$$
 and a similar  $E^4_{R,P}$.  This gives us  diagrams as in Lemma \ref{rabilemma}, with $B_1^1=(J_{Q,P}/J_{Q,P}^3)^*$,  $B_1^2= (J_{R,P}/J_{R,P}^3)^* $, $B_2^1= (J_{Q,P}/J_{Q,P}^4)^* $, $B_2^2= (J_{R,P}/J_{R,P}^4)^*$ and $A_1=H^1(C)$, $B_3=\otimes^2 H^1(C)$ and finally $C_1=\otimes^2 H^1(C)$. 

Let $\bff_Q$, $\bff_R$ and $\bff_{QR}$ denote the classes in $\Ext(\otimes^3 H^1(C),\otimes^2 H^1(C))$ with middle terms $H^{23}_{Q,P}, H^{23}_{R,P}$ and $H^{23}_{QR,P}$  corresponding to the diagrams for $Q$, $R$ and their generalised Baer difference.  $\bff_Q$ and $\bff_R$ are the push-forwards of $m^4_{Q,P}$ and $m^4_{R,P}$ respectively. From Corollary \ref{newextensions} one has 
$$\bff_{QR}=\bff_Q-\bff_R.$$

\begin{lem}\label{rabilemma3.3}  $\bff_{QR}$ is $N$-torsion in $\Ext(\otimes^3 H^1(C),\otimes^2 H^1(C))$. Namely,  
$$ N\cdot H^{23}_{QR,P}=\otimes^2 H^1(C) \oplus \otimes^3 H^1(C).$$
 \end{lem}
\begin{proof} From  Rabi \cite{rabi}, Corollary 3.3,  one has that the class $H^{23}_{Q,P}$ in $\Ext(\otimes^3 H^1(C),\otimes^2 H^1(C))$ is given by 
$$H^{23}_{Q,P}=H^{1}(C)\otimes E^{3}_{QP} \oplus_B E^{3}_{QP} \otimes H^1(C).$$
Similarly, 
$$H^{23}_{R,P}=H^{1}(C)\otimes E^{3}_{RP} \oplus_B E^{3}_{RP} \otimes H^1(C).$$
Taking their difference gives 
$$H^{23}_{R,P}-H^{23}_{Q,P}=H^{1}(C)\otimes(E^{3}_{R,P}-E^{3}_{Q,P}) \oplus_B (E^{3}_{RP}-E^{3}_{QP})\otimes H^{1}(C)$$
From Lemma \ref{splitting} we have  $E^{3}_{R,P}-E^{3}_{Q,P}=(0, 2g(Q-R)) \in \Ext_{MHS}(\otimes^{2}H^{1}(C),H^{1}(C))$. As $Q-R$ is $N$-torsion, we have  
$$N \cdot(H^{23}_{R,P}-H^{23}_{Q,P})\cong H^{1}(C)\oplus\otimes^{2}H^{1}(C).$$
  \end{proof}
 We also know from Lemma \ref{splitting} that $m^3_{QR,P}$ is $N$-torsion. Hence from Lemma \ref{rabilemma} we get an exact sequence 
 $$\begin{CD}
 0@>>> \otimes^2 H^1(C) \oplus H^1(C) @>>> N \cdot H^4_{QR,P} @>>> \otimes^3 H^1(C) \oplus \otimes^2 H^1(C) @>>> 0.
 \end{CD}$$ 
 which gives a class  in $\Ext( \otimes^3 H^1(C) \oplus  \otimes^2 H^1(C), \otimes^2 H^1(C) \oplus H^1(C))$. 
From the K\"unneth theorem, 
$$\Ext( \otimes^3 H^1(C) \oplus  \otimes^2 H^1(C), \otimes^2 H^1(C) \oplus H^1(C))= \prod_{ i \in \{2,3\}, j \in\{1,2\}}  \Ext(\otimes^i H^1(C),\otimes^j H^1(C)).$$
Define
\begin{equation}\label{fref}
e^4_{QR,P}  \in \Ext(\otimes^3 H^1(C), H^1(C))  
\end{equation}
to be the projection onto that component. Note that if $C$ is hyperelliptic, this class $e^4_{QR,P}$ is precisely the class $e^4_{QR,P}=e^4_{Q,P} \ominus_B e^4_{R,P}$ constructed by Colombo.

\subsection{Statement of the Main Theorem}

Armed with the class $e^4_{QR,P} \in \Ext(\otimes^3 H^1(C), H^1(C))$ we can proceed as in Colombo.

Let $\Omega$ denote the pullback of the polarisation on $J(C)$ in $H^2(J(C))(1)$ to $\otimes^2 H^1(C)(1)$. There is  an injection obtained by tensoring with $\Omega$,
$$J_{\Omega}=\otimes \Omega:H^1(C)(-1) \longrightarrow \otimes^3 H^1(C).$$
We first pull back the class using the map $J_{\Omega}$ to get a class in
$$J_{\Omega}^*(e^4_{QR,P}) \in \Ext(H^1(C)(-1),H^1(C)).$$
 Tensoring with $H^1(C)$ we get a class 
 $$J_{\Omega}^*(e^4_{QR,P}) \otimes H^1(C) \in  \Ext( \otimes^2 H^1(C)(-1), \otimes^2 H^1(C)).$$
 Once again pulling back using the map $\beta:\ZZ(-1) \rightarrow \otimes^2 H^1(C)$ gives us a class 
\begin{equation}\label{fref1}
\epsilon_{QR,P}^4 \in \Ext(\ZZ(-2),\otimes^2 H^1(C)) \subset \Ext(\ZZ(-2),H^2(C \times C))
\end{equation}

Our main theorem is 
 
 \begin{thm} Let $C$ be a smooth projective  curve of genus $g_C$ and $P$, $Q$ and $R$ be three distinct  points. Let  $Z_{QR,P}$ be the element of the motivic cohomology group $H^3_{\M}(J(C), \ZZ(2))$ constructed above. Let $\epsilon^4_{QR,P}$ be the extension in $\Ext_{MHS}(\ZZ(-2),\wedge^2 H^1(C))$ constructed above. Then 
 $$\epsilon^4_{QR,P}= (2g_C+1)N \reg_{\ZZ}(Z_{QR}) $$
 in $\Ext_{MHS}(\ZZ(-2),\wedge^2 H^1(C ))$.  
 \end{thm}  
 
In other words our theorem states that the regulator of a natural cycle in the motivic cohomology group of a product of curves, being thought of as an extension class  is the same as that as a natural extension of MHS coming from the fundamental group of the curve.  In fact, it is an extension of {\em pure} Hodge structures. 

\begin{rem}[Dependence on $P$] This is not so serious. If we do not normalise $f_{QR}$ with the condition that $f_{QR}(P)=1$ then one has to add an  expression of the form  $\log(f_{QR}(P)) \int_{C} \cdot$ to the term --- and this corresponds to adding a  {\em decomposable} element of the form $(\Delta_C,\log(f_{QR}(P)))$ to our element $Z_{QR}$. 

\end{rem}

\section{Carlson's Representatives}

The proof of the main theorem will follow by showing that both the algebraic cycle constructed earlier and the extension 
class constructed above  induce the same current. For that we have to understand the how an extension class induces a current. This comes from understanding the Carlson representative. In the section we once again follow Colombo \cite{colo} and adapt her arguments  to our situation.  

If $V$ is a MHS all of whose weights are negative, then the {\em Intermediate Jacobian of $V$} is defined  to be 
$$J_{0}(V) = \frac{ V_{\C}}{F^0V_{\C} \oplus V_{\ZZ}}.$$
This is a generalised torus - namely a group of the form $\C^a/\ZZ^b \simeq (\C^*)^{b} \times (\C)^{a-b}$ for some $a$ and $b$. 

An extension of mixed Hodge structures 
$$0 \longrightarrow A \stackrel{\iota}{\longrightarrow} H \stackrel{\pi}{\longrightarrow} B \longrightarrow 0$$
is called {\em separated}  if the lowest non-zero weight of $B$ is greater than the largest non-zero weight of $A$. This implies that $\Hom_{MHS}(B,A)$ has negative weights.  Carlson \cite{carl} showed that 
$$\Ext_{MHS}(B,A) \simeq J_0(\Hom(B,A)).$$
This is defined as follows. As an extension of {\em Abelian groups}, the extension splits. So  one has a map $r_{\ZZ}:H \rightarrow  A$ which is a retraction, $r_{\ZZ}\circ \iota=id$.  Let $s_F$ be a section in $\Hom(B_{\C},H_{\C})$ preserving the Hodge filtration. Then the Carlson representative of an extension is defined to be the class of
$$r_{\ZZ} \circ s_F \in J_0(Hom(B,A)).$$

We now describe explicitly the Carlson representative of the extension $\epsilon^4_{QR,P}$ constructed in the previous section. This is done in a few steps, first we describe the representative  for $e^4_{QR,P}$ and then  for its various pullbacks and push forwards to obtain that for $\epsilon^4_{QR,P}$. 

\subsection{Preliminaries}

As before, let $C$ be a smooth projective curve of genus $g=g_C$. We first describe  the Carlson representative of the extension 
$$e^4_{QR,P} \in \Ext_{MHS}(\otimes^3 H^1(C),H^1(C)).$$
From the above discussion this is an element of 
$$J_0(\Hom(\otimes^3 H^1(C),H^1(C))=\frac{\Hom(\otimes^3 H^1(C)_{\C},H^1(C)_{\C})}{F^0 \Hom(\otimes^3 H^1(C)_{\C},H^1(C)_{\C}) \oplus  \Hom(\otimes^3 H^1(C),H^1(C))}$$
so explicitly, given an element of $\otimes^3 H^1(C)_{\C}$ we get an element of $H^1(C)_{\C}$ which we can think of as a functional on $H_1(C)_{\C}$

Let $C_{QR}$ denote the open curve $C \setminus \{Q,R\}$. In fact we will describe the functional as an iterated integral made up of forms in  $H^1(C_{QR})_{\C}$ and will naturally be a functional on $H_1(C_{QR})$.  We have a natural inclusion 
$$i:C_{QR} \hookrightarrow C$$
which induces $i_*$ on homology and $i^*$ on cohomology.  In order to consider the iterated integral as a functional on $H_1(C)$ we have to make a choice of an embedding $H_1(C) \hookrightarrow H_1(C_{QR})$ which splits the map $i_*$. There are many ways of doing this, but for our formula to work, we need to make a particular choice. In this section we first construct a `natural' splitting of the map $i_*$ -- namely a subgroup of $H_1(C _{QR})$ which maps isomorphically to $H_1(C)$ under $i_*$. 

Consider the group $\pi_1(C_{QR};P)$. This is a free group on $2g+1$ generators. The generators have the following description. The fundamental polygon of $C$ is a $4g$ sided polygon with the edges $e_i$ and $e_{i+g}$ identified. The end  points of the edges are identified and so they give $2g$ loops $\alpha'_i$ in $C_{QR}$ which we consider as loops based at $P$. Let $\beta_Q$ be a small simple loop around $Q$ based at $P$. Then $\pi_1(C_{QR};P)=<\alpha'_1,\dots,\alpha'_{2g},\beta_Q>$. 

The map $f:C \rightarrow \CP^1$ restricts to give $f:C_{QR} \rightarrow \CP^1-\{0,\infty\}$ and this induces 
$$f_*:\pi_1(C_{QR};P) \longrightarrow  \pi_1( \CP^1-\{0,\infty\};1)$$
One knows $\pi_1(\CP^{1}-\{0,\infty\}) \simeq \ZZ$. Let $\beta_0$ denote the generator. Let $H=\Ker(f_*)$. $f_*(\pi_1(C_{QR};P))$ is a subgroup of $\ZZ$. In a deleted neighbourhood of $0$ the map looks like $z \rightarrow z^N$ where $N$ is the degree. Hence loop $\beta_Q$ is taken to $\beta_0^N$. Let $f_*(\alpha'_i)= \beta_0^{m_i}$ for some $m_i \in \ZZ$. Then $\alpha_i={\alpha'_i}^N \beta_Q^{-m_i}$ satisfies $f_*(\alpha_i)=0$.  Let $G$ denote the subgroup of $H=\Ker(f_*)$ generated by the $\{\alpha_i\}$.

The inclusion map $i$ also induces $i_*$ on the fundamental groups. Since $i_*(\beta_Q)=0$, $i_*(\alpha_i)=i_*(\alpha'_i)^N$. The fundamental group of $C$  is $\pi_1(C;P)=\{<i_*(\alpha'_1),\dots,i_*(\alpha'_{2g})>\slash \prod[i_*(\alpha'_i),i_*(\alpha'_{i+g})]=0\}$. Hence one has a map $G \rightarrow \pi_1(C;P)$ whose image is the subgroup generated by the $N^{th}$-powers of $\alpha'_i$. 

\begin{lem} The abelianization of $G$ is isomorphic to the subgroup of index $N^{2g}$ of the abelianization of $\pi_1(C;P)$.
$$G/[G,G] \simeq N \cdot \pi_1(C)/[\pi_1(C),\pi_1(C)]$$
where $N\cdot$ denotes multiplication by $N$. 
\label{commutator}
\end{lem}

\begin{proof}  Let $\alpha=\prod \alpha_{a_i}^{b_i}$ be a word in $G$. For a generator $\alpha_i$ of $G$ define 
$$\ord_{\alpha_i} (\alpha)=\sum_{a_i=i} b_i$$
namely, the number of times $\alpha_i$ appears in the word. Define
$$\Psi:G \rightarrow \ZZ^{2g}$$
$$\Psi(\alpha) = \left(\ord_{\alpha_i}(\alpha),\dots,\ord_{\alpha_{2g}}(\alpha)\right)$$
Let $K=\ker(\Psi)$. Clearly $[G,G] \subset K$. Further, the map $\Psi$ factors through $i_*$ and is surjective. 
We claim $K=[G,G]$. To see this, observe that if $a,b \in G$ 
$$ab \equiv ba \, \mod \,[G,G]$$
Repeatedly applying this one can see that any word 
$$\alpha=\prod \alpha_{a_i}^{b_i} \equiv \prod_{i=1}^{2g} \alpha_i^{\ord_{\alpha_i}(\alpha)} \, \mod \,[G,G]$$
In particular, if $\ord_{\alpha_i}(\alpha)=0$ for all $i$, $\alpha \in [G,G]$. Hence $K=[G,G]$. Hence 
$$G/[G,G] \simeq  \ZZ^{2g}$$
The map $i_*$ takes $\alpha_i$ to ${\alpha'}_i^N$. One has a similar map $\Psi':\pi_1(C;P) \rightarrow \ZZ^{2g}$ using $\alpha'_i$ instead of $\alpha_i$ which shows that the abelianization of $\pi_1(C;P)$ is $\ZZ^{2g}$ as well. However, under this map $\Psi'(\alpha_i)=N$ and hence $G/[G,G]$ is carried to the subgroup $N \cdot \ZZ^{2g}$. Multiplication by $N$ is an isomorphism so the map 
 $$i_N:=\frac{1}{N} \circ i_*:G/[G,G] \longrightarrow \pi_1(C)/[\pi_1(C),\pi_1(C)]$$
 is an isomorphism between the two abelianizations.

\end{proof}

Let $V=G/[G,G]$. The abelianization of the fundamental group of $C_{QR}$ is $H_1(C_{QR})$ and so $V$ is a subgroup of $H_1(C_{QR})$. The abelianization of $\pi_1(C)$ is $H_1(C)$. Hence the map $i_N$ is  an isomorphism between $V$  and $H_1(C)$.  Let $j_N:H_1(C) \longrightarrow V$  be the inverse isomorphism. This gives an embedding of $H_1(C)$ in $H_1(C_{QR})$.  As discussed above, the Carlson representative is a functional on  $H_1(C)$. However,  we will obtain a functional on $H_1(C_{QR})$ which will be the Carlson representative of the extension when considered as a functional on $V$.

Let $[\alpha]$ denote the homology class of a loop $\alpha$. The collection $\{[\alpha_i']\}$ has the property that their images $\{[i_*(\alpha'_i)]\}$ in $H_1(C)$ form a symplectic basis.  Since $i_*([\beta_Q])=0$, $i_*([\alpha_i])=N i_*([\alpha'_i])$. Hence under the isomorphism, $i_N([\alpha_i])=[i_*(\alpha'_i)]$. Let $\{dx_i\}$ be the dual  basis of harmonic forms in  $H^1(C,\C)$ satisfying
$$\int_{[i_*(\alpha_i')]} dx_j=\delta_{ij}$$
where $\delta_{ij}$ is the Kronecker Delta function. With this choice of $\{[\alpha'_i]\}s$ and $\{dx_i\}s$, the volume form on $H^2(C)$ can be expressed as follows. Let 
$$c(i)=\begin{cases} 1 &\text{ if } i\leq g_C\\
                                         -1 & \text{ if } i>g_C\end{cases}
                                                $$
and $\sigma(i)=i+c(i)g_C$.  The volume form is 
$$\sum_{i=1}^{2g_C} c(i) dx_i \wedge dx_{\sigma(i)}.$$
From that one gets that a Poincar\'e dual of $c(i)dx_{\sigma(i)}$ is  $[i_*(\alpha_i')]$. 

The group $V$ is a subgroup of $H_1(C_{QR})$. Recall that for the non-compact manifold $C_{QR}$,  Poincar\'e duality states that 
$$H^1_c(C_{QR}) \simeq H_1(C_{QR}),$$ 
where $H^1_c(C_{QR})$ is the cohomology with  compact support of $C_{QR}$.  This group has a mixed Hodge structure determined by identifying it with the relative cohomology group $H^1(C_{QR},\{Q,R\})$. Unlike $H^1(C_{QR})$ which has nontrivial weight $1$ and weight $2$ pieces, cohomology with compact support has weight $0$ and weight $1$ pieces and is covariant. However, we have an isomorphism of the weight $1$ graded pieces, 
$$Gr^{W}_1 H^1_c(C_{QR})_{\Q} \simeq Gr^W_1 H^1(C_{QR})_{\Q} \simeq H^1(C)_{\Q}$$
Here the first isomorphism is induced by $id_*$, the identity map  and the second by $i^*$. 

The space $V$ determines a splitting of the Hodge structure on $H_1(C_{QR})$. The space $V^*$ of Poincar\'e duals of element of $V$ is a subspace of $H^1_c(C_{QR})$ which determines a splitting of the Hodge structure on $H^1_c(C_{QR})$. Further, $V^*$ is isomorphic to $H^1(C)$. Hence if $\eta$ is a form in $H^1(C)$ it is cohomologous in $C_{QR}$ to a compactly supported form in $V^*  \subset H^1_c(C_{QR})$. One has 
$$H^1_c(V_{QR})_{\Q}=V_{\Q} \oplus \Q\cdot \omega_Q$$
where $\omega_Q$ is a Poincar\'e dual of $\beta_Q$. $\Q \cdot \omega_Q \simeq \Q(0)$. Note that 
$$\int_{[\alpha_j]} \omega_Q=\int_{C_{QR}} i^*(c(j)dx_{\sigma(j)})  \wedge \omega_Q=$$
$$= - \int_{C_{QR}} \omega_Q \wedge i^*(c(j) dx_{\sigma(j)}) = - \int_{\beta_Q} i^*(c(j)dx_{\sigma(j)})=-\int_{i_*([\beta_Q])} c(j) dx_{\sigma(j)}=0$$
since $i_*(\beta_Q)=0$.  Further 
$$\int_{[\alpha_i]} i^*(dx_j)=\int_{i_*([\alpha_i])} dx_j=\int_{Ni_*([\alpha'_i])} dx_j=N \delta_{ij}.$$
Hence the dual of $[\alpha_i]$ is $\frac{i^*(dx_i)}{N}$ and under the dual map $dx_i$ is taken to $\frac{dx_i}{N}$ in  $V^*=\Hom(V,\ZZ)$.  Further, note that 
$$\int_{[\alpha_i]} \frac{df}{f}=\int_{f_*(c(k)\alpha_{\sigma(k)})} \frac{dz}{z}=0$$
since $[\alpha_i] \in \Ker(f_*)$.
 
We now construct a cover of $C_{QR}$ which has the property that its first  homology group is $G/[G,G]$ and the form $\frac{df}{f}$ is {\em exact}.  Further, the loops $\alpha_i$ lift to loops on this cover. We do that as follows.  Let  $u:X \rightarrow C_{QR}$ denote the universal cover of $C_{QR}$. The group $G$ acts on $X$ as a group of deck transformations. Let $\tilde{C}=X/G$ denote the quotient and $q:\tilde{C} \rightarrow C_{QR}$ denote the covering map. This is a cover
\begin{equation}\label{cover}
q:(\tilde{C}, \tilde{P}) \longrightarrow (C_{QR},P) 
\end{equation}
such that $\pi_1(\tilde{C};\tilde{P}) = G$, where $\tilde{P}$ is a point in $q^{-1}(P)$. Now by homotopy lifting (\cite{hatc}, Proposition $1.31$), loops based at $P$ whose homotopy class lie in $G\subset \pi_1(C_{QR})$ will lift to loops in $\tilde{C}$ based at $\tilde{P}$. Thus $\alpha_i \in G$ will lift to a unique, upto homotopy loop $\tilde{\alpha}_i$ based at $\tilde{P}$ such that $q_*(\tilde{\alpha}_i)=\alpha_i$. 

\begin{prop}\label{coverlem}
$q^*(\frac{df}{f}) = 0$ in $H^1(\tilde{C})$. Hence there is a function, which we call $\log(q^*(f))$, defined on $\tilde{C}$ such that $d\log(q^*(f))=q^*(\frac{df}{f})$. 

\end{prop}
\begin{proof}  From Lemma \ref{commutator}, 
$$H_1(\tilde{C}) \simeq G/[G,G] \simeq V \simeq  N \cdot \pi_1(C)/[\pi_1(C;P),\pi_1(C;P)] \simeq N \cdot H_1(C) \simeq H_1(C)$$ 
By the de Rham isomorphism,  $H^1(\tilde{C}) \simeq H^1(C)$. 
 
The maps $q^*$ and $q_*$ are adjoint with respect to the de Rham pairing -- namely if $\sigma \in H_1(\tilde{C})$ and $\omega \in H^1(C_{QR},\C)$ then 
 $$\int_{q_*(\sigma)} \omega = \int_{\sigma} q^*(\omega).$$
 Further,  $q^*(\omega)$ is $0$  in $H^1(\tilde{C})$  if and only if $\int_{q_*(\sigma)} \omega=0$ for all $\sigma \in H_1(\tilde{C})$.  Applying this to $\omega=\frac{df}{f}$  and using the fact that $[\tilde{\alpha}_i], 1 \leq i \leq 2g$ give a basis for $H_1(\tilde{C})$,  we have 
$$q^*(\frac{df}{f})=0 \in H^1(\tilde{C})  \Leftrightarrow \int_{[\tilde{\alpha}_i]} q^*(\frac{df}{f})=0  \text { for all } i \Leftrightarrow  \int_{[\alpha_i]} \frac{df}{f}=0 \text{ for all } i.$$ 
The map $f$ induces 
$$f_*:H_1(C_{QR}) \longrightarrow H_1(\CP^1-\{0,\infty\}).$$
The form $\frac{df}{f}=f^*(\frac{dz}{z})$. Hence, one has 
$$\int_{[\alpha_i]} \frac{df}{f}=\int_{[\alpha_i]} f^*(\frac{dz}{z})=\int_{f_*([\alpha_i])} \frac{dz}{z}.$$
However, since $\alpha_i \in G$ and by choice $G \subset \Ker(f_*)$, we have $f_*(\alpha_i)=0$ so $f_*([\alpha_i])=0$ and finally  
$$\int_{f_*([\alpha_i])} \frac{dz}{z}=0$$
Hence $q^*(\frac{df}{f})=0 \in H^1(\tilde{C})$. Therefore integration of $\frac{df}{f}$ is path independent and we have a well defined function 
$$\log(q^*(f))(x)=\int_{\tilde{P}}^x q^*(\frac{df}{f})$$
on $\tilde{C}$. Note  that $\log(q^*(f)(\tilde{P}))=0$. 

\end{proof}

Hence the space $V$ can be understood as the homology of the space $\tilde{C}$ and the map $q_*$ gives a rational splitting of the map $i_*$.  We also have the following description of $V_{\Q}$. 

\begin{lem} Let $f:C_{QR} \longrightarrow \CP^1-\{0,\infty \}$ be the map with divisor $\div(f)=NQ-NR$ and $f(P)=1$ and $V=G/[G,G]$ as above. Let $W_{\Q}=\Ker(f_*:H^ 1(C_{QR})_{\Q} \longrightarrow H^1(\CP^1-\{0,\infty\}))$. Then $V_{\Q}=W_{\Q}$
\end{lem}

\begin{proof} Since $V \subset \Ker(f_*)$, $V_{\Q} \subset W_{\Q}$. However, both $V_{\Q}$ and $W_{\Q}$ are  subspaces of codimension $1$ in $H_1(C_{QR})_{\Q}$. Hence they are isomorphic. 
\end{proof}

Note that it does not appear to be true that $V=\Ker(f_*)$ as $\ZZ$-modules. Intrinsically, the reason why there is such a $V$ is the following. If $C$ and $C_{QR}$ are as above, there is an exact sequence of mixed Hodge structures 
$$ 0 \longrightarrow \ZZ(1) \longrightarrow H_1(C_{QR}) \longrightarrow H_1(C) \longrightarrow 0$$
induced by the inclusion map. Hence $H_1(C_{QR})$ determines a class in $\Ext(H_1(C),\ZZ(1))$. From the Carlson isomorphism one knows 
$$\Ext_{MHS} (H_1(C),\ZZ(1)) \simeq \Ext_{MHS}(\ZZ(-1),H^1(C)) \simeq CH^1_{hom}(C)$$
and the class determined by  $H_1(C_{QR})$ is nothing but the class of $Q-R$ in $CH^1_{hom}(C)$. Since there exists a function $f$ with $\div(f)=NQ-NR$ it implies that this sequence splits rationally. Hence there is a map 
$$p:H_1(C_{QR})_{\Q} \longrightarrow  \Q(1)$$ 
which splits the exact sequence. This map can be seen to be 
$$p(\sigma)= \int_{\sigma}  \frac{df}{f} = \int_{f_*(\sigma)} \frac{dz}{z} $$
and if $V_{\Q}$ is the kernel, then $V_{\Q} \simeq H_1(C)_{\Q}$. Clearly $\sigma \in \Ker(p) \Leftrightarrow \sigma \in \Ker(f_*)$.  Hence $\Ker(f_*)$ is isomorphic to $H_1(C)_{\Q}$. 
The  $V$ defined above is only contained in $\Ker(f_*)$  but is a subgroup of  the integral homology $H_1(C)$ -- so has a little more information.

\subsection{The Carlson Representative of $e^4_{QR,P}$}

The Carlson representative of $e^4_{QR,P}$  is given by 
$$p_1 \circ r_{\ZZ} \circ s_F \circ i_3,$$
where 

\begin{itemize}
\item  $p_1$ is the projection of $N \cdot  H^3_{QR,P} \simeq H^1(C) \oplus \otimes^2 H^1(C)  \stackrel{p_1}{\longrightarrow} H^1(C)$.
\item  $i_3$ is the inclusion map $\otimes^3 H^1(C) \stackrel{i_3}{\hookrightarrow} \otimes^3 H^1(C) \oplus \otimes^2 H^1(C)$. 
\end{itemize}
To describe $s_F$ we need a little more. Let $C_{\bullet}=C-\{\bullet\}$ for $\bullet \in \{Q,R\}$. The inclusion map 
$$i_{\bullet}:C_{\bullet} \hookrightarrow C$$
induces isomorphisms on the first homology and cohomology groups -- and in what follows we will identify elements of $H_1(C_{\bullet})$  with their images in $H_1(C)$ and similarly elements of  $H^1(C)$ with their images in $H^1(C_{\bullet})$.

Recall  $\tilde{\ominus}_B$ denotes the generalised Baer difference. Let 
$$s_F \circ i_3: \otimes^3 H^1(C) \longrightarrow N\cdot H^4_{QR,P} \simeq  N\cdot \left( (J_{Q,P}/J_{Q,P}^4)^*   \tilde{\ominus}_B (J_{R,P}/J_{R,P}^4)^* \right)$$
 be the  section preserving the Hodge filtration given by 
 $$s_F(dx_i \otimes dx_j \otimes dx_k )=(I^{ijk}_Q,I^{ijk}_R).$$
Here  $I^{ijk}_{\bullet} \in (J_{\bullet,P}/J_{\bullet,P}^4)^*$ for $\bullet \in \{Q,R\}$  are iterated integrals with 
\begin{equation}\label{iteratedformula}
I^{ijk}_{\bullet}= N \left(\int dx_i dx_j dx_k +  dx_i \mu_{jk,\bullet}+\mu_{ij,\bullet}dx_k + \mu_{ijk,\bullet}  \right).
\end{equation}
where $\mu_{ij,\bullet}$,  $\mu_{jk,\bullet}$ and $\mu_{ijk,\bullet}$ are smooth, logarithmic $1$-forms on $C_{\bullet}$ such that 
\begin{enumerate}\label{it123}
\item $d\mu_{jk,\bullet}+dx_j \wedge dx_k=0$
\item $d\mu_{ij,\bullet}+dx_i \wedge dx_{j}=0$
\item $dx_i \wedge \mu_{jk,\bullet} + \mu_{ij,\bullet}\wedge dx_k +d\mu_{ijk,\bullet}=0.$
\end{enumerate}

There are inclusion maps of $C_{QR}$ into $C_{Q}$ and $C_{R}$ and we can pull back the forms $dx_i$, $\mu_{ij,\bullet}$ and $\mu_{ijk,\bullet}$ to $C_{QR}$ and consider all the forms as forms on $C_{QR}$.   To compute the element of $\Hom(\otimes^3 H^1(C)_{\C}, H^1(C)_{\C})$ obtained as the projection under $p_1$,  we describe it as an element of $H_1(C)_{\C}^*=\Hom(H_1(C),\C)=H^1(C,\C)$.  The integrands $I^{ijk}_{\bullet}$ are made up of forms on $C_{QR}$ and so to compute it on elements of $H_1(C)$ we have to choose an embedding of $H_1(C)$ in $H_1(C_{QR})$. This is precisely what the subgroup $V$ gives us. 

Hence from now on if $\alpha$ is a homology class in $H_1(C)$ we think of it as an element of $H_1(C_{QR})$ by identifying it with its image in $V=q_*(H_1(\tilde{C}))$. Let $V^*$ denote its dual in $H^1_c(C_{QR})$. The Poincar\'e dual of a element of $V$ also lies in $V^*$.

The map from 
$$H^1(C) \longrightarrow (H^1(C) \oplus H^1(C))/\Delta_{H^1(C)}$$ 
is given by 
$$x \longrightarrow (x,-x).$$ 
Further, if $\alpha$ is a loop based at $P$ on $C_{QR}$,   the class in $H_1(C)=J_{\bullet,P}/J_{\bullet,P}^{2}$ corresponding to it is $1-\alpha$. So one has $p_1 \circ r_{\ZZ} \circ s_F \circ i_3 \in \Hom(\otimes^3 H^1(C)_{\C},H^1(C)_{\C})$. As an integral, it is 
$$p_1 \circ r_{\ZZ} \circ s_F \circ i_3(dx_i\otimes dx_j\otimes dx_k)(\alpha)=\int_{1-\alpha} I^{ijk}_Q - \int_{1-\alpha} I^{ijk}_R$$
where the first $1-\alpha$ is the class in $H_1(C_Q)$ and the second is the class in $H_1(C_R)$. They are both carried to the same class in $V$ under the isomorphism, so we can take the difference of the integrals when we consider $\alpha$ as a loop in $C_{QR}$ whose corresponding homology class lies in $V$. This resulting expression is 
$$\int_{1-\alpha} I^{ijk}_Q - \int_{1-\alpha} I^{ijk}_R= N \left(\int_{1-\alpha} dx_i \left(\mu_{jk,Q}-\mu_{jk,R}\right)+\left(\mu_{ij,Q} -\mu_{ij,R}\right) dx_k + \left(\mu_{ijk,Q}-\mu_{ijk,R}\right) \right).$$  
We can choose the logarithmic forms 
$\mu_{ij,\bullet}$ and $\mu_{ijk,\bullet}$, for $\bullet \in \{Q,R\}$,  satisfying the following 
\begin{itemize}
\item $\mu_{ij,\bullet}=-\mu_{ji,\bullet}$.
\item For $|i-j|\neq g_{C}$, $\mu_{ij,\bullet}$ is  smooth on $C$,  as $d\mu_{ij,\bullet}=dx_j \wedge dx_i=0$. As $H^2(C_{QR},\ZZ)=0$ and $\mu_{ij,\bullet}$  is  smooth, it is  orthogonal to all closed forms, that is, $\mu_{ij,\bullet}\wedge dx_k$ is exact. 
\item $\mu_{i\sigma(i),\bullet}$  has a logarithmic singularity at $\bullet$ with residue $c(i)$. 
\item $\mu_{ij,Q}-\mu_{ij,R}=0$ if $|i-j| \neq g_C$ as forms on $C_{QR}$. 
\item $\mu_{i\sigma(i),Q}-\mu_{i\sigma(i),R}=\frac{c(i)}{N}\frac{df}{f}$, where  $f=f_{QR}$ is a function such that $\div(f_{QR})=NQ-NR$. We can normalise $f_{QR}$ once again by requiring that $f_{QR}(P)=1$. 
\label{properties}
\end{itemize}

In terms of our basis of  forms of  $H^1(C)$,  $\Omega \in \otimes^2 H^1(C)$ is 
$$\Omega=\sum_{i=1}^{g_C} dx_i \otimes dx_{(i+g_C)} - dx_{(i+g_C)} \otimes dx_i=\sum_{i=1}^{2g_C} c(i)dx_i \otimes dx_{\sigma(i)}$$
With these choices of $\mu_{ij,\bullet}$ and $\mu_{ijk,\bullet}$, we have the following theorem:

\begin{thm}\label{imthm} 
Let  ${\bf G}_{QR,P} \in \Hom(H^1(C)(-1)_{\C},H^1(C)_{\C})$ be the Carlson representative  corresponding to the extension class $J_{\Omega}^*(e^4_{QR,P})$. It is given by 
$${\bf G}_{QR,P}(dx_k)(\alpha_j)=p_1 \circ r_{\ZZ} \circ s_F \circ i_3( dx_k \otimes  \Omega) (\alpha_j)= (2g_C+1)  \int_{\alpha_j}  \frac{df}{f}dx_k - N \int_{\alpha_j} W(dx_k).$$
in $J(\Hom(H^1(C)(-1),H^1(C)))$, where 
$$W(dx_k)=\sum_{i=1}^{2g_C} c(i)( \mu_{k i \sigma(i),Q} -\mu_{ki\sigma(i),R})$$
is a $1$-form on $C_{QR}$ which satisfies 
$$d W(dx_k)=(2g+1) \frac{dx_k}{N} \wedge \frac{df}{f}$$
\end{thm}

\begin{proof} Let $S_F$ denote the map $S_F=s_F \circ i_3 \circ J_{\Omega}:H^1(C)(-1) \rightarrow N \cdot H^4_{QR,P}$. This is given by 
$$S_F(dx_k)=\sum_{i=1}^{2g_C} c(i) s_F( dx_k \otimes dx_i \otimes dx_{\sigma(i)})$$
From \eqref{iteratedformula} one has 
$$S_F(dx_k)=\left( \sum_{i=1}^{2g_C} c(i) \int I_Q^{ki\sigma(i) },\sum_{i=1}^{2g_C} c(i) \int  I_R^{ki\sigma(i)} \right)$$
Evaluating on a loop $\alpha_j$ based at $P$ using the maps described above, this is 
$$\sum_{i=1}^{2g_C} \int_{1-\alpha_j} c(i)   \left( I_Q^{ki\sigma(i)} -  I_R^{k i \sigma(i)} \right)$$
$$\sum_{i=1}^{2g_C} N  \left( \int_{1-\alpha_j} c(i) dx_k (\mu_{i\sigma(i),Q} - \mu_{i,\sigma(i),R}) + (\mu_{ki,Q}-\mu_{ki,R}) dx_{\sigma(i)} + (\mu_{k i \sigma(i),Q}-\mu_{k i \sigma(i),R})\right)$$
From the choice of the forms $\mu_{ij,\bullet}$ and $\mu_{ijk,\bullet}$ above,  the leading terms and several of the lower order terms cancel out and
$$ \mu_{ki,Q}-\mu_{ki,R}=c(k) \delta_{k\sigma(i)} \frac{1}{N} \frac{df}{f} $$
 and 
 $$ \mu_{i\sigma(i),Q}-\mu_{i\sigma(i),R}=c(i)\frac{1}{N}\frac{df}{f}.$$
 Since $c(i)^2=1$ what remains is 
$$   \sum_{i=1}^{2g_C}  \int_{1-\alpha_j} dx_k \frac{ df}{f}  -  \int_{1-\alpha_j} \frac{df}{f} dx_k  +  N \sum_{i=1}^{2g_C} c(i) \int_{1-\alpha_j} \left(\mu_{k i \sigma(i),Q}-\mu_{k i \sigma(i),R} \right).$$
Let 
$$W(dx_k)=  \sum_{i=1}^{2g_C} c(i)   \left(\mu_{k i \sigma(i),Q}-\mu_{k i \sigma(i),R} \right).$$
Since integration over a point, which corresponds to the constant loop 1, is $0$ and $\int_{\alpha_j} \frac{df}{f}=0$ by choice of $\alpha_j$, using Lemma \ref{basicproperties} (2)  the integral becomes 
$${\bf G}_{QR,P}(dx_k)(\alpha_j)=    2g_C \int_{1-\alpha_j} dx_k \frac{df}{f} - \int_{1-\alpha_j} \frac{df}{f} dx_k  + N \int_{1-\alpha_j} W(dx_k).$$ 
$$=-(2g_C+1)\int_{\alpha_j}  dx_k \frac{df}{f}  -  N  \int_{\alpha_j}W(dx_k).  $$
Now consider 
$$dW(dx_k)= \sum_{i=1}^{2g_C} c(i)  d \left(\mu_{k i \sigma(i),Q}-\mu_{k i \sigma(i),R} \right).$$
From the choice of $\mu_{ijk,\bullet}$, one has 
$$d\mu_{ijk,\bullet}=-dx_i \wedge \mu_{jk,\bullet} - \mu_{ij,\bullet}\wedge dx_k.$$
So the sum becomes 
$$dW(dx_k)=\sum_{i=1}^{2g_C}  -c(i)\left( \left(dx_k \wedge \mu_{i\sigma(i),Q} +\mu_{ki,Q} \wedge dx_{\sigma(i)} \right) -\left( dx_k \wedge \mu_{i\sigma(i),R} + \mu_{ki,R} \wedge dx_{\sigma(i)} \right)\right) $$
$$=\sum_{i=1}^{2g_C} -c(i) \left(  dx_k\wedge (\mu_{i\sigma(i),Q}-\mu_{i\sigma(i),R}) + (\mu_{ki,Q}-\mu_{ki,R}) \wedge dx_{\sigma(i)} \right).$$
In the second sum, only one term survives and one has 
$$=-c(\sigma(k))(\mu_{k\sigma(k),Q}-\mu_{k\sigma(k),R}) \wedge dx_k + \sum_{i=1}^{2g_C} -c(i) \left(dx_k \wedge \frac{c(i)}{N} \frac{df}{f}\right)$$
$$=-c(\sigma(k))(\frac{c(\sigma(k))}{N} \frac{df}{f}) \wedge dx_k + \sum_{i=1}^{2g_C} -c(i) \left( dx_k \wedge \frac{c(i)}{N} \frac{df}{f}\right)$$
$$=-\frac{(2g_C+1)}{N} \frac{df}{f} \wedge dx_k=\frac{(2g_C+1)}{N} dx_k \wedge \frac{df}{f}.$$

\end{proof}

We have computed the Carlson representative $\bf{G}_{QR,P}$ of our class  in $\Ext(H^1(C)(-1),H^1(C))$.  We now tensor with $H^1(C)$ and pull back using the map $\otimes \Omega:\ZZ(-1)\longrightarrow \otimes^2 H^1(C)$.  This gives us an element  of $\Ext(\ZZ(-2), \otimes^2 H^1(C))$. We denote its Carlson representative by ${\bf F}_{QR,P}$.                                                  
\begin{lem}\label{fqrplemma} The Carlson representative of the class in $\Ext(\ZZ(-2),\otimes^2 H^1(C))$ is given by 

$${\bf F}_{QR,P} = ({\bf G}_{QR,P} \otimes Id) \circ \otimes \Omega$$ 
in $(\otimes^2 H^1(C)_{\C})^*$. On an element $\alpha_j \otimes \alpha_k$ 
 it is given by
\begin{equation}
{\bf F}_{QR,P}(\Omega)(\alpha_j \otimes \alpha_k)=c(\sigma(k))N \left(\int_{\alpha_j} (2g_C+1)\frac{df}{f}dx_{\sigma(k)} - NW(dx_{\sigma(k)})\right) 
\label{fqrp}
\end{equation}
\end{lem}

\begin{proof} Recall that 
$$\Omega=\sum_1^{2g_C} c(i)dx_i \otimes dx_{\sigma(i)}.$$
From above we have 
$$({\bf G}_{QR,P} \otimes Id )(\Omega)(\alpha_j \otimes \alpha_k) = \sum_1^{2g_C}  c(i) {\bf G}_{QR,P}(dx_i)(\alpha_j) \cdot  Id(dx_{\sigma(i)}) (\alpha_k).$$
From the choice of $\alpha_k$ one has 
$$Id(dx_{\sigma(i)}) (\alpha_k) =N\delta_{k\sigma(i)}.$$
Hence, in the sum above,  precisely one term survives, at $i=\sigma(k)$. Therefore 
$$({\bf G}_{QR,P} \otimes Id ) (\Omega)(\alpha_j \otimes \alpha_k)=N c( \sigma(k)) {\bf G}_{QR,P}(dx_{\sigma(k)})(\alpha_j).$$
In particular
\begin{align*}
{\bf F}_{QR,P}(\Omega)(\alpha_j \otimes \alpha_k)=& N c(\sigma(k)) {\bf G}_{QR,P}(dx_{\sigma(k)})(\alpha_j)\\
=&N c(\sigma(k))\left(\int_{\alpha_j} (2g_C+1)\frac{df}{f}dx_{\sigma(k)} - NW(dx_{\sigma(k)}) \right) 
\end{align*}

\end{proof}

We now use Proposition \ref{coverlem} to convert the iterated integral in to an ordinary integral.   The iterated integral term in \eqref{fqrp} is 
$$Nc(\sigma(k)) (2g_C+1)\int_{\alpha_j} \frac{df}{f} dx_k$$
which we can evaluate using Lemma \ref{basicproperties}(3) if $\frac{df}{f}$ is exact. However, $\frac{df}{f}$ is {\em not} exact on $C_{QR}$ but it is exact on $\tilde{C}$ using Proposition \ref{coverlem}.  So we do the integration on  $\tilde{C}$. Precisely, we do that as follows.

Let $\alpha$ be a loop such that $[\alpha] \in q_*(H_1(\tilde{C}))$, where $q:\tilde{C} \longrightarrow C_{QR}$ is the cover.  Let $\alpha=q_*(\tilde{\alpha})$, where $\tilde{\alpha}$ is a loop based at $\tilde{P}$ lying over the basepoint $P$ of $\alpha$. Let $\psi$ be another closed $1$-form on $C_{QR}$. 

We have 
$$\int_{\alpha} \frac{df}{f} \psi=\int_{\tilde{\alpha}} q^{*}(\frac{df}{f}) q^*(\psi)$$
From Proposition \ref{coverlem},   $q^*(\frac{df}{f})$ is exact on $\tilde{\alpha}$. In other words $q^*(\frac{df}{f})=d\log(q^*(f))$. Choose a primitive $\log(q^*(f))$ such that $\log(q^*(f)(\tilde{P}))=0$.  Using Lemma \ref{basicproperties}(3) and the fact that we have chosen $\log(q^*(f))$ with $\log(q^*(f)(\tilde{P}))=0$,
$$\int_{\tilde{\alpha}} q^*({\frac{df}{f}})q^*(\psi)= \int_{\tilde{\alpha}} \log(q^*(f))q^*(\psi)$$
Hence we have 
$$\int_{\alpha} \frac{df}{f} \psi=\int_{\tilde{\alpha}} \log(q^*(f))q^*(\psi)$$
Applying this to the case at hand  we have 
\begin{equation}
\label{eq123}
{\bf F}_{QR,P}(\Omega)(\alpha_j \otimes \alpha_k)=N c(\sigma(k))\left(\int_{\tilde{\alpha}_j} (2g_C+1) \log(q^*(f))q^*(dx_\sigma(k))- Nq^*(W(dx_{\sigma(k)}))\right).
\end{equation}
We have made a choice of $\tilde{\alpha_j}$. If we chose a different basepoint, the value of $\log(q^*(f))$ will change by $2 \pi i M$ for some $M \in \ZZ$. This will change the integral by  $2 \pi i M \int_{\alpha_j} dx_{\sigma(k)}$.  This does not affect the class in the intermediate Jacobian. 

We would like to connect the expression above, which is the Carlson representative of the  extension class $\epsilon^4_{QR,P}$, to the regulator of an explicit cycle on the Jacobian of the curve. This is done by the following revised version of Colombo's Proposition 3.3. Recall that by $C\setminus \gamma$ we mean
$$\lim_{\epsilon \rightarrow 0}  C \setminus \ngam$$
where $\ngam$ is an open tubular neighbourhood of $\gamma$ homeomorphic to $(-\epsilon,\epsilon) \times \gamma$. 

\begin{prop}[Colombo Proposition 3.3]
Let  $f=f_{QR}$ be as before and  $\psi$ a  closed $1$-form on $C_{QR}$. Let $W(\psi)$ be a $1$-form such that  $dW(\psi)=\psi \wedge \frac{df}{f}$. Hence $\Theta=\log(f)\psi+W(\psi)$ is a closed $1$ form on $C_{QR} \setminus \gamma$. Let $\alpha$  be a loop in $C_{QR}$ such that $[\alpha] \in V$ and  let $\eta_{\alpha} \in H^1_c(C_{QR})$ be the Poincar\'e dual of $[\alpha]$ constructed below. Then we have the iterated integral 
$$ \int_{\alpha} \frac{df}{f} \psi + W(\psi) =\int_{C \setminus \gamma} \eta_{\alpha} \wedge \Theta + 2\pi i  \int_{\gamma} \eta_{\alpha} \psi \,  \left(\mod\, 2\pi i \int_{\alpha} \psi \ZZ\right)$$
\label{colomboprop}
\end{prop}

\begin{proof} 

We first need the following useful lemma.
\begin{lem} Let $\Theta$ and $\psi$ be as above. Then 
$$\int_{\gamma} \psi=0$$
\label{usefullemma}
\end{lem}

\begin{proof} Let $C \setminus \ngam$ be the manifold with boundary as before. This is a closed set so the  form $\Theta$ is closed and compactly supported  on $C \setminus \ngam$. The boundary 
$$\partial(C \setminus \ngam)=Q \times (-\epsilon, \epsilon) \cup \gamma_1 \cup \gamma_2^{-1} \cup (\epsilon,-\epsilon) \times R$$
where $\gamma_1$ and $\gamma_2$ are copies of $\gamma$. Applying Stokes theorem to $d\Theta$ we get
$$0=\int_{C\setminus \ngam} d\Theta = \int_{\partial(C \setminus \ngam)} \Theta= \int_{\gamma_1} \Theta - \int_{\gamma_2} \Theta + \int_{Q \times (-\epsilon, \epsilon)} \Theta - \int_{R \times (-\epsilon, \epsilon)} \Theta $$
Recall $\Theta=\log(f) \psi + W(\psi)$. The function $\log(f)$ differs by $2 \pi i $ on  $\gamma_1$ and $\gamma_2$. The form $W(\psi)$ is defined on all of $C_{QR}$ so the integrals over $\gamma_1$ and $\gamma_2$ cancel. Keeping track of orientations we have 
$$0= -2 \pi i \int_{\gamma_1} \psi +  \int_{Q \times (-\epsilon, \epsilon)} \Theta - \int_{R \times (-\epsilon, \epsilon)} \Theta $$
As $\epsilon \rightarrow 0$ the shows that $ -2 \pi i \int_{\gamma} \psi=0.$
\end{proof}

The subgroup $V$ is generated by the classes of $\alpha_i={\alpha'}_i^N \beta_Q^{-m_i}$ where $\alpha'_i$ is one of the `standard' generators of  $\pi_1(C)$ coming from the edges of the fundamental polygon and $\beta_Q$ is a small simple loop around $Q$. These loops satisfy $f_*(\alpha_i)=0$. 

It suffices to prove the theorem for the $\alpha_i$ and extend linearly, so from this point on we  let $\alpha=\alpha_i$, $\alpha'=\alpha'_i$. Let $\eta_{\alpha}$ be the compactly supported  Poincar\'{e} dual of $[\alpha]$ constructed as in \cite{fakr} as follows. Suppose $\delta$ is a simple loop in $C_{QR}$. Let $D=D_{\delta}$ be a tubular neighbourhood of $\delta$. We can write $D_{\delta}-\delta=D_{\delta}^+ \cup D_{\delta}^-$ with $D_{\delta}^{-}$ to the left and $D_{\delta}^{+}$ to the right of $\delta$.  Let $D_0$ be a sub-tubular neighbourhood of $\delta$ in  $D$ and $D_0^{\pm}=D_0  \cap D_{\delta}^{\pm}$. Let $G_{\delta}$ be a function such that is smooth on $C_{QR}-\delta$ and 
$$G_{\delta} \equiv  \begin{cases}  1 & \text{ on } D_0^- \cup {\delta} \\
 0 &\text{ outside } D_{\delta}^{-}\end{cases}$$ 
Define 
$$\eta_{\delta}=\begin{cases} dG_{\delta} & \text{ on $D_{\delta}-\delta$}\\0 & \text{ elsewhere}  \end{cases}$$
so the support $\Supp(\eta_{\delta}) \subset D_{\delta}^-$. One can then see that if $\psi$ is a  closed $1$-form on $C_{QR}$
$$\int_{C_{QR}} \eta_{\delta} \wedge \psi = \int_{D_{\delta}^-} dG_{\delta} \wedge \psi = \int_{D_{\delta}^-} dG_{\delta} \psi =  \int_{\partial (D_{\delta}^-)} G_{\delta} \psi =  \int_{[\delta]} \psi$$
since $G_{\delta} \equiv 1$ on $\delta$ and with this choice of orientation $\partial(D_{\delta}^-)=\delta$. 

In our case, in general $\alpha=\alpha_i={\alpha'}^N \beta_Q^{-m_i}$, is not a {\em simple} loop. However, $\alpha'$ and $\beta_Q$ are.  Let $D_{\alpha}^-=D_{\alpha'}^- \cup D_{\beta_Q}^-$. Define 
$$\eta_{\alpha}=N\eta_{\alpha'}-m_i \eta_{\beta_Q}.$$
$\eta_{\alpha}$ is supported in $D_{\alpha}^-$ and is a Poincar\'{e} dual of $[\alpha]$ as, for a $1$-form $\psi$, 
 $$\int_{C_{QR}} \eta_{\alpha} \wedge \psi =N \int_{C_{QR}}  \eta_{\alpha'} \wedge \psi  - m_i \int_{C_{QR}} \eta_{\beta_Q} \wedge \psi =\int_{N[\alpha'] - m_i [\beta_Q]} \psi = \int_{[\alpha]} \psi$$

Let $\tilde{\Theta}=q^*(\Theta)=\log(q^*(f))q^*(\psi) +q^*(W(\psi)) $. From the discussion above 
$$\int_{\alpha} \frac{df}{f} \psi + W(\psi)= \int_{\tilde{\alpha}} \tilde{\Theta}.$$
where $\tilde{\alpha}$ is a lifting of $\alpha$ to a loop in $\tilde{C}$ such that it is based at $\tilde{P}$ and $\log(q^*(f))$ is chosen such that $\log(q^*(f)(\tilde{P})=0.$ 
We would like to compute this integral. 

Let $\ngam$ be as above and  $C \setminus \ngam$ be as before. Choosing $\epsilon$ and the tubular neighbourhoods carefully, we can assume without loss of generality that $\alpha'$ and $\beta_Q$ do not pass through the points $Q$ and $R$ and that the tubular neighbourhoods $D_{\alpha'}$ and $D_{\beta_Q}$  do not intersect $(-\epsilon,\epsilon) \times Q$ and $(-\epsilon,\epsilon) \times R$.  

We have $\alpha=\alpha'^N \beta_Q^{-m_i}$. Let $\tilde{\alpha}'$ be the lift of $\alpha'$ and $\tilde{\beta}_Q$ the lift of $\beta_Q$. The restriction $\alpha'|_{C_{QR} \setminus \ngam}=\bigcup_j \alpha'^j$ is a union of a finite number of  paths $\alpha'^j$. The covering map $q$ induces a homeomorphism from each  $\alpha'^j$  to a path $\tilde{\alpha}'^j$ such that $\bigcup \tilde{\alpha}'^j=\tilde{\alpha}'|_{\tilde{C} \setminus q^{-1}(\ngam)}$. Let $D_{\alpha'^j}$, and similarly $D_{\alpha'^j}^{-}$  denote the restriction of the tubular neighbourhood $D_{\alpha'}$  of $\alpha'$ to a tubular neighbourhood of the path $\alpha'^j$. 
We have 
$$D_{\alpha'}^- \setminus  \ngam = \bigcup_j D_{\alpha'^j}^{-}$$
Hence we have 
$$\int_{C_{QR} \setminus \ngam} \eta_{\alpha'} \wedge \Theta = \int_{D_{\alpha}'^- \setminus  \ngam}  \eta_{\alpha'} \wedge \Theta=\sum_{j}  \int_{D_{\alpha'^j}^-}  \eta_{\alpha'} \wedge \Theta$$
The boundary of the  tubular neighbourhood $D_{\alpha'^j}^-$  is 
 $$\partial(D_{\alpha'^j}^-)=\alpha'^j \cup ({\gamma_1}\cap \overline{D_{\alpha'^j}^-})   \cup (\gamma_2^{-1} \cap \overline{D_{\alpha'^j}^-})$$ 
 Applying Stokes' Theorem we get
 $$\int_{D_{\alpha'^j}^-}  \eta_{\alpha'} \wedge \Theta = \int_{D_{\alpha'^j}^-} dG_{\alpha'}\Theta = \int_{\partial(D_{\alpha'^j}^-)} G_{\alpha'}\Theta$$
 $$= \int_{\alpha'^j} \Theta + \int_{\gamma_1 \cap \overline{D_{\alpha'^j}^-}} G_{\alpha'} \Theta - \int_{\gamma_2 \cap \overline{D_{\alpha'^j}^-}} G_{\alpha'}\Theta  $$ 
Summing up over $j$ we have
 $$ \sum_j \left(\int_{\gamma_1 \cap \overline{D_{\alpha'^j}^-}} G_{\alpha'} \Theta - \int_{\gamma_2 \cap \overline{D_{\alpha'^j}^-}} G_{\alpha'}\Theta \right)=\int_{\gamma_1} G_{\alpha'} \Theta -\int_{\gamma_2} G_{\alpha'} \Theta$$
 as $G_{\alpha'}$ is supported in $D_{\alpha'}^-$. Recall that the value of $\log(f)$ on $\gamma_1$ and $\gamma_2$ differ by $-2\pi i$. The value of $G_{\alpha'} W(\psi)$ is the same on both $\gamma_1$ and $\gamma_2$ hence $\int_{\gamma_1} G_{\alpha'} W(\psi)-\int_{\gamma_2} G_{\alpha'} W(\psi)=0$.  Hence we get 
 $$\int_{\gamma_1} G_{\alpha'} \Theta - \int_{\gamma_2} G_{\alpha'} \Theta = -2 \pi i \int_{\gamma} G_{\alpha'} \psi $$
Lemma \ref{usefullemma} and  Lemma \ref{basicproperties} (3) shows that the integral simplifies to 
$$\int_{\gamma_1} G_{\alpha'} \Theta - \int_{\gamma_2} G_{\alpha'} \Theta = -2 \pi i \int_{\gamma} G_{\alpha'} \psi = -2 \pi i \int_{\gamma} \eta_{\alpha'} \psi$$
So we  get 
$$  \int_{C_{QR} \setminus \ngam} \eta_{\alpha'} \wedge \Theta=\sum_j \int_{\alpha'^j} \Theta - 2 \pi i \int_{\gamma} \eta_{\alpha'} \psi$$
We can make a similar argument for $\beta_Q$. 

Combining the two we get 
$$N \sum_j \int_{\alpha'^j} \Theta - m_i  \sum_s \int_{\beta_Q^s}  \Theta = \int_{C_{QR}  \setminus \ngam} \eta_{\alpha} \wedge \Theta + 2 \pi i  \int_{\gamma} \eta_{\alpha}\psi$$
Finally, since $\eta_{\alpha}$ is compactly supported in $C_{QR}$, we can replace $C_{QR} \setminus \ngam$ with $C\setminus \ngam$ to get 
$$N \sum_j \int_{\alpha'^j} \Theta - m_i  \sum_s \int_{\beta_Q^s}  \Theta = \int_{C  \setminus \gamma} \eta_{\alpha} \wedge \Theta + 2 \pi i  \int_{\gamma} \eta_{\alpha}\psi$$

To link this to the integral over $\tilde{\alpha}$ we observe the following. The loop $\alpha'^N$ lifts to a path in $\widetilde{\alpha'^N}$ in $\tilde{C}$ which is made up of copies of $\tilde{\alpha}'$. Let $\tilde{\alpha}'_k$ denote the lift of the $k^{th}$ copy of $\alpha'$ so $\tilde{\alpha}'_k(1)=\tilde{\alpha}'_{k+1}(0)$. We can choose the homeomorphisms between $\alpha'^j$ and $\tilde{\alpha}'^j$ such that the $k^{th}$ copy of $\alpha'^j$ is homeomorphic to a path $\tilde{\alpha}'^j_k$ in $\tilde{\alpha}'_k$. So we have homeomorphisms 
$$\bigcup_{k=1}^N \bigcup_j \alpha'^j \simeq \bigcup_{k=1}^{N} \bigcup_j \tilde{\alpha}'^j_k \simeq \widetilde{\alpha'^N} - \partial(\tilde{C} \setminus q^{-1}(\ngam))$$
A similar situation holds for $\beta_Q$. Via these homeomorphisms 
$$N \sum_j \int_{\alpha'^j} \Theta - m_i  \sum_s \int_{\beta_Q^s}  \Theta =\sum_{k=1}^N \sum_j \int_{\tilde{\alpha}'^j_k} \tilde{\Theta} -  \sum_{r=1}^{m_i} \sum_s \int_{\tilde{\beta}_{Q,r}^s}  \tilde{\Theta} $$
which is 
$$\int_{\widetilde{\alpha'^N \beta_Q^{-m_i}}  - \partial(\tilde{C} \setminus q^{-1}(\ngam))} \tilde{\Theta}=\int_{\tilde{\alpha}-\partial(\tilde{C} \setminus q^{-1}(\ngam))} \tilde{\Theta}$$
Finally, as $\epsilon \rightarrow 0$ the set  $\tilde{\alpha} \cap \partial(\tilde{C} \setminus q^{-1}(\ngam))$ becomes a set of measure $0$ so we have 
 $$\lim_{\epsilon \rightarrow 0} \int_{\tilde{\alpha}-\partial(\tilde{C} \setminus q^{-1}(\ngam))} \tilde{\Theta}= \int_{\tilde{\alpha}} \tilde{\Theta}$$ 
We have made a choice of a lifting $\tilde{\alpha}$  of $\alpha$. A different choice of basepoint $\tilde{P}'$ would change the value of $\log(q^*(f))$ by $2 \pi i M$ for some $M \in \ZZ$. This would change the integral by $2 \pi i M \int_{\alpha}  \psi$. Hence this equality holds only up to $2 \pi i \int_{\alpha} \psi \ZZ$.

 \end{proof}

We have the following useful corollary to the above proposition, which says that in fact, we can replace $\eta_{\alpha}$ by any form on $C$ which is cohomologous to $i_*([\eta_{\alpha}])$. 

\begin{cor} Let $f$, $\psi$, and $W(\psi)$  be as above and $\phi_{\alpha}$ a closed $1$ form on $C$ which  is cohomologous in $H^1(C)$ to $i_*([\eta_{\alpha}])$ for some $\alpha$ in $V$. Then
$$ \int_{\alpha} \frac{df}{f} \psi + W(\psi) =\int_{C \setminus  \gamma} \phi_{\alpha} \wedge \Theta + 2\pi i  \int_{\gamma} \phi_{\alpha} \psi \,  \left(\mod\, 2\pi i \int_{\alpha} \psi \ZZ\right)$$
\label{nocompact}
\end{cor}

\begin{proof} Let $\phi_{\alpha}$ denote such a form. Both $\phi_{\alpha}$ and $\eta_{\alpha}$ are Poincar\'e duals of the same homology class, so $\phi_{\alpha}-\eta_{\alpha}=dg$. Let $\ngam$ be as before. One has 
$$\int_{C\setminus \ngam} \phi_{\alpha} \wedge \Theta  = \int_{C \setminus \ngam} dg \wedge \Theta + \int_{C \setminus \ngam} \eta_{\alpha} \wedge \Theta$$
So it suffices to compute the two terms separately. 

Any closed form $\phi$ on $C$ is compactly supported on the manifold with boundary $C \setminus \ngam$ as $C \setminus \ngam$ is a closed subset of $C$ and $C$ is compact.  Further, by choice of $V$ it will be cohomologous to $[i_*(\eta_{\alpha})]$ for some $\alpha \in V$.  In particular we know that $[c(i)\frac{dx_{\sigma(i)}}{N}]$ is cohomologous to a Poincar\'e dual of $[\alpha_i]$.

Since both $\phi_{\alpha}$ and $\eta_{\alpha}$ are  compactly supported on $C \setminus \ngam$, so is   $dg$ and hence $g$ and  $dg \wedge \Theta$ are compactly supported as well. From Stokes' Theorem we get 
 $$\int_{C \setminus \ngam} dg \wedge \Theta = \int_{C \setminus \ngam}d(g\Theta)= \int_{\partial (C \setminus \ngam)} g\Theta$$
 If we choose a different function  $h$ such that $dh=dg$, then $h-g=c$ for some constant $c$. Hence, from Lemma \ref{usefullemma} we see
 $$ \int_{\partial (C \setminus \ngam)} g\Theta-  \int_{\partial (C \setminus \ngam)} h\Theta=  c \int_{\partial (C \setminus \ngam)} \Theta = \int_{C\setminus \ngam} d\Theta=0$$  
  so it does not depend on the choice of primitive. 
  
 An argument similar to Lemma \ref{usefullemma} with $g\Theta$ in the place of $\Theta$  shows that 
 $$\int_{\partial (C \setminus \ngam)} g\Theta = -2\pi i \int_{\gamma} g\psi$$
 Lemma \ref{usefullemma} along with Lemma \ref{basicproperties} (3) shows further that 
  $$ \int_{C \setminus \ngam} dg \wedge \Theta = \int_{\partial (C \setminus \ngam)} g\Theta = -2\pi i \int_{\gamma}  g\psi= -2\pi i \int_{\gamma} dg \psi.$$
 Taking the limit as $\epsilon \rightarrow 0$ one gets
$$ \int_{C \setminus \gamma} dg \wedge \Theta = -2\pi i \int_{\gamma} dg \psi.$$ 
From Proposition \ref{colomboprop} we have 
$$ \int_{\alpha} \frac{df}{f} \psi + W(\psi) =\int_{C \setminus \gamma} \eta_{\alpha} \wedge \Theta + 2\pi i  \int_{\gamma} \eta_{\alpha} \psi \,  \left(\mod\, 2\pi i \int_{\alpha} \psi \ZZ\right)$$
Since the integral of $dg \wedge \Theta$ cancels the iterated integral term, we have 

$$ \int_{C \setminus \gamma} \phi_{\alpha} \wedge \Theta + 2\pi i  \int_{\gamma} \phi_{\alpha} \psi =\int_{C\setminus \gamma} \left(\eta_{\alpha} \wedge \Theta + dg \wedge \Theta\right) + 2 \pi i \int_{\gamma} \left( \eta_{\alpha} \psi + dg \psi\right) $$
$$= \int_{C \setminus \gamma} \eta_{\alpha} \wedge \Theta + 2 \pi i  \int_{\gamma} \eta_{\alpha} \psi = \int_{\alpha} \frac{df}{f} \psi + W(\psi)   \left(\mod\, 2\pi i \int_{\alpha} \psi \ZZ\right)$$

\end{proof}

\begin{rem} We will apply this to compute the Carlson representative of an extension class. This is an element of the intermediate Jacobian associated to the extension, hence the two expressions, while they possibly differ by and element of $(2 \pi i \int_{\alpha} \psi )\ZZ $, will have the same class in   $\Hom_{\ZZ} (\ZZ(-2),\otimes^2 H^1(C))$ keeping in mind isomorphism  of $H_1(C)$ with $q_*(H_1(\tilde{C})) \subset H_1(C_{QR})$. We will use $\equiv$ to keep in mind the fact that all equalities hold only in the intermediate Jacobian. 
\end{rem}

We now apply this in the case of interest to us.
\begin{cor}                                                                                                                                                                                                                                                                                                                                                                                                                                                                                                                                                                                                                                                                                                                                                                                                                                                                                                                                                                                                                                                                                                                                                                                                                                                                                                                                                                                                                                                                                                                                                                                                                                                                                                                                                                                                                                                                                                                                                                                                                                                                                                                                                                                                                                                                                                                                                                                                                                                                                                                                                                                                                                                                                                                                                                                                                                                                                                                                                                                                                                                                                                                                                                                                                                                                                                                                                                                                                                                                                                                                                                                                                                                                                                                                                                                                                                                                                                                                                                                                                                                                                                                                                                                                                                                     
Let  $\alpha_j$ be as above. Then                                                                                                                                                                                                                                                                                                                                                                                                                                                                                                                                                                                                                                                                                                                                                                                                                                                                                                                                                                                                                                                                                                                                                                                                                                                                                                                                                                                                                                                                                                                                                                                                                                                                                                                                                                                                                                                                                                                                                                                                                                                                                                                                                                                                                                                                                                                                                                                                                                                                                                                                                                                                                                                                                                                                                                                                                                                                                                                                                                                                                                                                                                                                                                                                                                                                                                                                                     \begin{align*}
{\bf F}_{QR,P}(\Omega)(\alpha_j \otimes \alpha_k)  \equiv &N(2g_C+1)c(j)c(\sigma(k)) \left(\int_{C\setminus \gamma}  dx_{\sigma(j)} \wedge \left( \log(f)dx_{\sigma(k)}-\frac{N}{(2g_C+1)}W(dx_{\sigma(k)})\right) \right. \\
 &+ \left. 2\pi i \int_{\gamma}  dx_{\sigma(j)} dx_{\sigma(k)} \right) 
\end{align*}
 
\end{cor}
\begin{proof} This is a straightforward application of Corollary  \ref{nocompact} to the expression in Lemma \ref{fqrplemma}.

\end{proof}  
${\bf F}_{QR,P}(\Omega)$ determines an element of the intermediate Jacobian of $(\otimes^2 V_{\C})^*$
$$J(\otimes^2 V_{\C}^*) \simeq \frac{F^1(\otimes^2 V_{\C}^*)}{(\otimes^2 V)^*}$$
so to determine ${\bf F}_{QR,P}(\Omega)$ it suffices to evaluate it on elements of $F^1(\otimes^2 V_{\C})^*$.  From the above Corollary, we can compute it on $F^1(\otimes^2 H^1(C))$.

Let $\{dz_i\}$ be a basis of the holomorphic $1$-forms on $C$  such that 
$$\int_{i_*(\alpha_i)} dz_j=N\delta_{ij} \hspace{1in} 1\leq i\leq g$$ 
 where $\{\alpha_i\}$ is the basis $V$.
 $$dz_j=dx_{j}+ \sum_{i=1}^{g} A_{ji} dx_{i+g} \hspace{1in} \text { where } A_{ji}=\frac{1}{N} \int_{i_*(\alpha_{i+g})} dz_j$$ 
 coming from the fact that $c(j)\frac{dx_{\sigma(j)}}{N}$ is dual to $\alpha_j$. This expression continues to hold when we consider them as forms in $C_{QR}$.
 
Let $\zeta_{j}=c(\sigma(j)) \alpha_{\sigma(j)}+\sum_{1\leq i \leq g} A_{ji} c(i) \alpha_{i},$ where $j\leq g$. Then $dz_j$ is cohomologous in $H^1(C_{QR})$ to a Poincar\'e dual of $\zeta_j$. We  have the following proposition
\begin{prop}\label{colombolemma}The map ${\bf F}_{QR,P}(\Omega)$ evaluated on elements of the form $\zeta_i \otimes \alpha_j$ is 
 $${\bf F}_{QR,P}(\Omega)(\zeta_i \otimes c(\sigma(j))\alpha_{\sigma(j)}) \equiv (2g_C+1)N \left(\int_{C \setminus \gamma}  \log(f) dz_i \wedge  dx_j  + 2\pi i \int_{\gamma}  dz_i dx_j \right) $$
In other words 
$$dz_i \wedge W(dx_j)=0.$$
\end{prop}

\begin{proof} This is Colombo, \cite{colo}. Proposition 3.4. The point is that $W(dx_j)$ and $dz_i$ are both of type $(1,0)$ hence $dz_i \wedge W(dx_j)=0$ 

\end{proof}

In fact, the theorem holds for the other term as well. 

\begin{prop} For a suitable choice of $\mu_{ijk,Q}$ and $\mu_{ijk,R}$ one has 
\begin{align*}W(dz_{i}):=&W(dx_{i})+\sum_{k}A_{ki}W(dx_{i+g})=0 
\end{align*}
\label{wdz}
\end{prop}
\begin{proof} \cite{colo} Lemma $3.1$.
\end{proof}
Hence we have 
\begin{prop}\label{wd1} 
\label{integralformula}
$${\bf F}_{QR,P}(\Omega)(c(\sigma(j))\alpha_{\sigma(j)} \otimes \zeta_i) \equiv (2g_C+1) N \left(\int_{C \setminus \gamma}  \log(f) dx_j \wedge  dz_i + 2\pi i \int_{\gamma}  dx_j dz_i \right).$$
\end{prop}
Comparing this with the regulator term in Theorem \ref{regform}  we get 
\begin{thm} Let $Z_{QR}$ be the motivic cohomology cycle constructed above  and $\epsilon_{QR,P}^4$ the extension in $\Ext_{MHS}(\ZZ(-2),\wedge^2 H^1(C))$. We use $\epsilon^4_{QR,P}$ to denote its Carlson representative as well.  Then one has 
$$ \epsilon^4_{QR,P}(\omega) \equiv (2g_C+1)N \reg_{\ZZ}(Z_{QR})(\omega)$$
where $\omega \in F^1\wedge^2 H^1(C)$. 
\label{mainthm}
\end{thm}
\begin{proof} It suffices to check this on $dz_i \wedge dx_j=dz_i \otimes dx_j-dx_j \otimes dz_i$. The result then follows by comparing 
the formula for the  Carlson representative ${\bf F}_{QR,P}$ in Lemma \ref{fqrplemma} with the expression for the regulator in Theorem  
\ref{regform} using Proposition \ref{colombolemma}

From Theorem \ref{integralformula}  and Lemma \ref{fqrplemma} we have
\begin{align*}
{\bf F}_{QR,P}(\Omega)(c(\sigma(j))\alpha_{\sigma(j)} \otimes \zeta_i) \equiv &(2g_C+1)N \left(\int_{C \setminus \gamma}  \log(f) dx_{j}  \wedge dz_i    + 2\pi i \int_{\gamma}  dx_j d z_j \right)\\
\end{align*}

On the other hand, from  Propostion \ref{wd1}  one has 
\begin{align*}
{\bf F}_{QR,P}(\Omega)(\zeta_i \otimes c(\sigma(j))\alpha_{\sigma(j)})&=(2g_C+1)N\left(\int_{C\setminus \gamma}  \log(f) dz_i \wedge  dx_j  + 2\pi i \int_{\gamma}  dz_i dx_j \right)\\
&=(2g_C+1)N\left(-\int_{C \setminus \gamma}  \log(f) dx_j \wedge  dz_i  + 2\pi i \int_{\gamma}  dz_idx_j \right)
\end{align*}
Therefore we get 
\begin{align*}
{\bf F}_{QR,P}(\Omega)(c(\sigma(j))\alpha_{\sigma(j)} \wedge \zeta_i) \equiv & {\bf F}_{QR,P}(\Omega)(c(\sigma(j))\alpha_{\sigma(j)} \otimes \zeta_i) - {\bf F}_{QR,P}(\Omega)(\zeta_i \otimes c(\sigma(j))\alpha_{\alpha_{\sigma(j)}})\\
\equiv & (2g_C+1)N\left(2\int_{C \setminus \gamma}  \log(f) dx_{j}  \wedge dz_i    + 2\pi i \int_{\gamma} (dx_j dz_i - dz_i dx_j)\right) \\
\end{align*}
On the other hand, from Theorem \ref{regform}
$$(2g_C+1)N\reg_{\ZZ}(Z_{QR})( dx_j\wedge dz_i) =(2g_C+1)N\left(2\int_{C \setminus \gamma}  \log(f) dx_{j}  \wedge dz_i    + 2\pi i \int_{\gamma} (dx_jdz_i-dz_i dx_j)\right)$$
\end{proof}

Recall that we have assumed in both cases  that $f_{QR}(P)=1$. If we do not make that assumption, then one has a term corresponding to a decomposable element that one has to account for. However, if we work modulo the decomposable cycles we can ignore that term.   

As a result of this theorem, we get the following expression of the regulator as an iterated integral over a loop - which is more amenable to computation. 

\begin{cor}

Let $Z_{QR,P}$ be the element of $H^{2g-1}_{\M}(\J(C),\ZZ(g))$ and let $\eta$ and $\omega$  be two closed $1$ forms on $C$ with $\omega$ holomorphic. Let $\alpha$ be loop in $C_{QR}$ based at $P$ such that $[\alpha] \in V$ and a   Poincar\'e dual of $[\alpha]$ is homologous to $\eta$ in $H^1(C_{QR})$. Then 

%
$$\reg_{\ZZ}(Z_{QR,P})(\eta \wedge \omega)  \equiv 2(2g_C+1)N \int_{\alpha} \frac{df}{f}\omega.$$
\end{cor}
\begin{proof}  It suffices to check this for $\eta=dx_j$ and $\omega=dz_i$. From  Theorem \ref{mainthm}, one has 
$${\bf F}_{QR,P}(\Omega)(c(\sigma(j))\alpha_{\sigma(j)} \wedge \zeta_i) \equiv (2g_C+1) N \left(2 \int_{C \setminus \gamma}  \log(f) dx_j \wedge  dz_i + 2\pi i \int_{\gamma}  (dx_j dz_i -dz_i dx_j) \right).$$
From Lemma \ref{basicproperties} we have
$$\int_{\gamma} dx_j dz_i + \int_{\gamma} dz_i dx_j = \int_{\gamma} dx_j \int_{\gamma} dz_i$$
From Lemma \ref{usefullemma} one has  $\int_{\gamma} dz_i = 0$. Hence 
$$\int_{\gamma} dx_j dz_i = - \int_{\gamma} dz_i dx_j$$
Therefore, we can simplify the regulator expression to get 
$$\reg_{\ZZ}(Z_{QR,P})(dx_{j}\wedge dz_{i} )= 2(2g_C+1)N\left(\int_{C \setminus \gamma}  \log(f) dx_j  \wedge dz_i    + 2\pi i \int_{\gamma} dx_jdz_i \right)$$
$$=2(2g_C+1)N \left(\int_{C \setminus \gamma} dx_j \wedge \log(f)dz_i + 2\pi i \int_{\gamma} dx_j dx_i \right)$$
Using Corollary  \ref{nocompact}  this becomes 
$$\reg_{\ZZ}(Z_{QR,P})(dx_{j}\wedge dz_{i} )= 2(2g_C+1) c(j) N  \int_{\alpha} \frac{df}{f} dz_i.$$

\end{proof}

\begin{rem} Darmol-Rotger-Sols \cite{DRS} and Otsubo \cite{otsu2} have a similar formula for the regulator of the modified diagonal cycle. Their expression is an iterated integral of two holomorphic forms over the dual of a third forn - but none of the forms are logarithmic. Similarly, the regulator of an element of $K_2$ can be expressed as the iterated integral of two logarithmic forms over the dual of a third holomorphic form. 

\end{rem}

\bibliographystyle{alpha}
\bibliography{Extensions.bib}

\begin{thebibliography}{DRS12}

\bibitem[Blo84]{blocspec}
Spencer Bloch.
\newblock Algebraic cycles and values of {$L$}-functions.
\newblock {\em J. Reine Angew. Math.}, 350:94--108, 1984.

\bibitem[Blo00]{blocirvine}
Spencer~J. Bloch.
\newblock {\em Higher regulators, algebraic {$K$}-theory, and zeta functions of
  elliptic curves}, volume~11 of {\em CRM Monograph Series}.
\newblock American Mathematical Society, Providence, RI, 2000.

\bibitem[Car80]{carl}
James~A. Carlson.
\newblock Extensions of mixed {H}odge structures.
\newblock In {\em Journ\'ees de {G}\'eometrie {A}lg\'ebrique d'{A}ngers,
  {J}uillet 1979/{A}lgebraic {G}eometry, {A}ngers, 1979}, pages 107--127.
  Sijthoff \& Noordhoff, Alphen aan den Rijn, 1980.

\bibitem[Che77]{chen}
Kuo~Tsai Chen.
\newblock Iterated path integrals.
\newblock {\em Bull. Amer. Math. Soc.}, 83(5):831--879, 1977.

\bibitem[Col97]{coll}
A.~Collino.
\newblock Griffiths' infinitesimal invariant and higher {$K$}-theory on
  hyperelliptic {J}acobians.
\newblock {\em J. Algebraic Geom.}, 6(3):393--415, 1997.

\bibitem[Col02]{colo}
Elisabetta Colombo.
\newblock The mixed {H}odge structure on the fundamental group of hyperelliptic
  curves and higher cycles.
\newblock {\em J. Algebraic Geom.}, 11(4):761--790, 2002.

\bibitem[DRS12]{DRS}
Henri Darmon, Victor Rotger, and Ignacio Sols.
\newblock Iterated integrals, diagonal cycles and rational points on elliptic
  curves.
\newblock In {\em Publications math\'ematiques de {B}esan\c con. {A}lg\`ebre et
  th\'eorie des nombres, 2012/2}, volume 2012/ of {\em Publ. Math. Besan\c con
  Alg\`ebre Th\'eorie Nr.}, pages 19--46. Presses Univ. Franche-Comt\'e,
  Besan\c con, 2012.

\bibitem[FK92]{fakr}
H.~M. Farkas and I.~Kra.
\newblock {\em Riemann surfaces}, volume~71 of {\em Graduate Texts in
  Mathematics}.
\newblock Springer-Verlag, New York, second edition, 1992.

\bibitem[Hai87]{hain}
Richard~M. Hain.
\newblock The geometry of the mixed {H}odge structure on the fundamental group.
\newblock In {\em Algebraic geometry, {B}owdoin, 1985 ({B}runswick, {M}aine,
  1985)}, volume~46 of {\em Proc. Sympos. Pure Math.}, pages 247--282. Amer.
  Math. Soc., Providence, RI, 1987.

\bibitem[Har83]{harr}
Bruno Harris.
\newblock Homological versus algebraic equivalence in a {J}acobian.
\newblock {\em Proc. Nat. Acad. Sci. U.S.A.}, 80(4 i.):1157--1158, 1983.

\bibitem[Hat02]{hatc}
Allen Hatcher.
\newblock {\em Algebraic topology}.
\newblock Cambridge University Press, Cambridge, 2002.

\bibitem[IM14]{iymu}
J.~N. {Iyer} and S.~{M{\"u}ller-Stach}.
\newblock {Degeneration of the modified diagonal cycle}.
\newblock {\em ArXiv e-prints}, January 2014.

\bibitem[Kae01]{kaen}
Rainer~H. Kaenders.
\newblock The mixed {H}odge structure on the fundamental group of a punctured
  {R}iemann surface.
\newblock {\em Proc. Amer. Math. Soc.}, 129(5):1271--1281, 2001.

\bibitem[Ots11]{otsu1}
Noriyuki Otsubo.
\newblock On the regulator of {F}ermat motives and generalized hypergeometric
  functions.
\newblock {\em J. Reine Angew. Math.}, 660:27--82, 2011.

\bibitem[Ots12]{otsu2}
Noriyuki Otsubo.
\newblock On the {A}bel-{J}acobi maps of {F}ermat {J}acobians.
\newblock {\em Math. Z.}, 270(1-2):423--444, 2012.

\bibitem[Pul88]{pult}
Michael~J. Pulte.
\newblock The fundamental group of a {R}iemann surface: mixed {H}odge
  structures and algebraic cycles.
\newblock {\em Duke Math. J.}, 57(3):721--760, 1988.

\bibitem[Rab96]{rabihodge}
Reuben Rabi.
\newblock A note on the hodge theory of curves and periods of integrals.
\newblock {\em preprint}, pages 1--71, 1996.

\bibitem[Rab01]{rabi}
Reuben Rabi.
\newblock Some variants of the logarithm.
\newblock {\em Manuscripta Math.}, 105(4):425--469, 2001.

\bibitem[Sch93]{scho2}
A.~J. Scholl.
\newblock Extensions of motives, higher {C}how groups and special values of
  {$L$}-functions.
\newblock In {\em S\'eminaire de {T}h\'eorie des {N}ombres, {P}aris, 1991--92},
  volume 116 of {\em Progr. Math.}, pages 279--292. Birkh\"auser Boston,
  Boston, MA, 1993.

\end{thebibliography}

 \begin{tabular}[t]{l@{\extracolsep{8em}}l} 
Ramesh Sreekantan  & Subham Sarkar \\
Statistics and Mathematics Unit & School of Mathematics\\
Indian Statistical Institute & Tata Institute of Fundamental Research\\
8th Mile, Mysore Road & 1 Homi Bhabha Road \\
Jnanabharathi & Colaba\\
Bengaluru 560 059 & Mumbai 400 005\\ 
rameshsreekantan@gmail.com  & subham.sarkar13@gmail.com
\end{tabular}

\end{document}